\documentclass[11pt,a4paper]{amsart}
\usepackage{amsmath,amssymb, amsbsy}
\usepackage{color,psfrag}
\usepackage[dvips]{graphicx}
\usepackage{color}
\usepackage{mathtools}
\theoremstyle{plain}
\usepackage{enumerate}
\newtheorem{theorem}{Theorem}[section]
\newtheorem{proposition}[theorem]{Proposition}%[section]
\newtheorem{lemma}[theorem]{Lemma}%[section]
\newtheorem{remark}[theorem]{Remark}%[section]
\newtheorem{definition}[theorem]{Definition}%[section]

\def\endproof{\hfill $\Box$ \par \vskip3mm}

\newcommand{\R}{\mathbb{R}}
\newcommand{\rn}{\mathbb{R}^{N}}

\newcommand{\hn}{\mathbb{H}^{N}}

\newcommand{\ds}{\: {\rm d}s}

\newcommand{\gradg}{\nabla_g}
\newcommand{\dvg}{{\rm d}V_g}
\newcommand{\dr}{\: {\rm d}r}

\numberwithin{equation}{section} \allowdisplaybreaks

\usepackage[text={6in,8.6in},centering]{geometry}
\begin{document}
	
	\title[Semilinear equations on Riemannian models]{Classification of radial solutions to\\ $-\Delta_g u=e^u$ on Riemannian models}

	\author[Elvise BERCHIO]{Elvise BERCHIO}
\address{\hbox{\parbox{5.7in}{\medskip\noindent{Dipartimento di Scienze Matematiche, \\
Politecnico di Torino,\\
       Corso Duca degli Abruzzi 24, 10129 Torino, Italy. \\[3pt]
       \em{E-mail address: }{\tt elvise.berchio@polito.it}}}}}

       \author[Alberto FERRERO]{Alberto FERRERO}
\address{\hbox{\parbox{5.7in}{\medskip\noindent{ Dipartimento di Scienze e Innovazione Tecnologica, \\
Universit\'a del Piemonte Orientale,\\       Viale Teresa Michel 11, Alessandria, 15121, Italy. \\[3pt]
 \em{E-mail address: }{\tt alberto.ferrero@uniupo.it }}}}}
	
	\author[Debdip Ganguly]{Debdip Ganguly}
	\address{\hbox{\parbox{5.7in}{\medskip\noindent{Department of Mathematics,\\
					Indian Institute of Technology Delhi,\\
					IIT Campus, Hauz Khas, Delhi,\\
					New Delhi 110016, India. \\[3pt]
					\em{E-mail address: }{\tt
						debdipmath@gmail.com}}}}}
					
	\author[Prasun Roychowdhury]{Prasun Roychowdhury}
	\address{\hbox{\parbox{5.7in}{\medskip\noindent{ Mathematics Division,\\
					National Center for Theoretical Sciences,\\
					NTU, Cosmology Building, No.~1, Sec. 4\\
					 Roosevelt RD, Taipei City 106, Taiwan. \\[3pt]
					\em{E-mail address: }{\tt prasunroychowdhury1994@gmail.com}}}}}

	\date{\today}
\begin{abstract}

We provide a complete classification with respect to asymptotic behaviour, stability and intersections properties of radial smooth solutions to the equation $-\Delta_g u=e^u$ on Riemannian model manifolds $(M,g)$ in dimension $N\ge 2$. Our assumptions include Riemannian manifolds with sectional curvatures bounded or unbounded from below. Intersection and stability properties of radial solutions are influenced by the dimension $N$ in the sense that two different kinds of behaviour occur when $2\le N\le 9$ or $N\ge 10$, respectively. The crucial role of these dimensions in classifying solutions is well-known in Euclidean space; here the analysis highlights new properties of solutions that cannot be observed in the flat case.

\end{abstract}

	\maketitle

\section{Introduction}

Let $N \ge 2$ and let $M$ be an $N$-dimensional \textit{Riemannian model} $(M,g)$, namely a manifold admitting a pole $o$
and whose metric is given, in polar or spherical coordinates around $o$, by
\begin{equation}\label{metric}
	g = {\rm d}r^2 + (\psi(r))^2 \, {\rm d}\omega^2, \qquad  r>0, \, \omega \in\mathbb{S}^{N-1} \, ,
\end{equation}
for some function $\psi$ satisfying suitable assumptions. Here ${\rm d}\omega^2$ denotes the canonical metric on the unit sphere $\mathbb{S}^{N-1}$ and $r$ is by construction the distance between a point of spherical coordinates $(r,\omega)$ and the pole $o$.

In this article, we are concerned with radial smooth solutions of
\begin{align} \label{main_eq}
       -\Delta_{g}u \, =\, e^u \quad	\text{ in } M \, ,
\end{align}
namely smooth solutions of \eqref{main_eq} depending only on the geodesic distance from the pole. Here $\Delta_g$ denotes the Laplace-Beltrami operator in $(M,g)$.

%\rosso The great interest for equation \eqref{main_eq} is motivated by its applications to geometry and physics see e.g., \cite{cheng-ni,edi,gelfand}. For instance, in a fluid dynamic, \eqref{main_eq} posed on rotating spheres arises when studying ``Stuart vortices", i.e., exact stationary solutions for which the vorticity is assumed to be exponentially related to the streamfunction, see \cite{crowdy}. Related semilinear equations on rotating spheres also arise in the study of stratospheric planetary stationary flows, see \cite{constantin} and the discussion therein.
% \nero
% 
%
%\bigskip
%
%\bigskip
%
%\blu 
The great interest for equation \eqref{main_eq} is motivated by its applications to geometry and physics see e.g., \cite{cheng-ni,edi,gelfand}. For instance, in fluid dynamics, a variant of \eqref{main_eq} appears when studying ``Stuart vortices" on a unit sphere $\mathbb S^2$ in a stationary regime, i.e.
$$
   -\Delta_{\mathbb S^2} \, \psi=b +  c \, e^{d\psi}  \qquad \text{on } \mathbb S^2
$$ 
where $\psi$ represents the stream function, $b,c,d$ are real constants and the right hand side of the equation represents a vorticity $\omega$ depending exponentially on the streamfunction $\psi$, see \cite{crowdy}. Related semilinear equations on rotating spheres also arise in the study of stratospheric planetary stationary flows, see \cite{constantin} and the discussion therein. When looking for stationary stream functions, the authors of \cite{constantin} are led to consider the semilinear elliptic equation
$$
  \Delta_{\mathbb S^2} \, \psi=F(\psi)-2\omega \sin \theta \qquad \text{on } \mathbb S^2 
$$
where now $\omega$ represents the rotation speed around the polar axis, $F$ is a smooth function and $\theta$ is the angle of latitude. 

The behaviour of solutions to equation \eqref{main_eq}, when posed in the Euclidean space, has been fully understood from different points of view: asymptotic behaviour, stability and intersections properties. In the seminal paper \cite{joseph-73} existence and asymptotic behaviour of radial solutions were established by means of a dynamical system analysis based on the so-called Emden transformation of the phase plane. We refer to \cite[Theorem 1.1]{tello-06} for the intersection properties of solutions, while the classification of solutions with respect to stability is given in \cite{dancer,farina2}. See also \cite{DUFAneq} and references therein for more recent results in the case of general nonlinearities. Finally, related results in the weighted case: $-\Delta u=k(x)e^u$ in $\R^n$ can be found in \cite{bae-18,cheng-lin-87,cheng-lin-92}, extensions to the higher order case are instead given in \cite{BFFG} and in \cite{faFE}.

%The great interest for equation \eqref{main_eq} when posed in the Euclidean space is motivated by its applications to geometry and physics, see e.g., \cite{cheng-ni,edi,gelfand}. In this framework, the behaviour of solutions has been fully understood from different points of view: asymptotic behaviour, stability and intersections properties. In the seminal paper \cite{joseph-73} existence and asymptotic behaviour of radial solutions were established by means of a dynamical system analysis based on the so-called Emden transformation of the phase plane. We refer to \cite[Theorem 1.1]{tello-06} for the intersections properties of solutions, while the classification of solutions with respect to stability is given in \cite{dancer,farina2}. See also \cite{DUFAneq} and references therein for more recent results in the case of general nonlinearities. Finally, related results in the weighted case: $-\Delta u=k(x)e^u$ in $\R^n$ can be found in \cite{bae-18,cheng-lin-87,cheng-lin-92}, extensions to the higher order case are instead given in \cite{BFFG} and in \cite{faFE}.

As already pointed out in \cite{joseph-73}, equation \eqref{main_eq} can be regarded as the limit case as $p \rightarrow + \infty$ of the Lane-Emden-Fowler equation: $-\Delta_g u=|u|^{p-1}u$ in $M$ ($p>1$). In the last ten years, there has been an intense study of this equation in non-euclidean frameworks including the hyperbolic space and more general Riemannian models, see \cite{bandle-12,elvise-14,bonforte-13,GS,hasegawa,muratori,sandeep-mancini} and references therein. In these papers existence, multiplicity, asymptotic and stability results were provided. The analysis settled on general manifolds highlights deep relationships between the qualitative behaviour of solutions and intrinsic properties of the manifold itself and, sometimes, it reveals a number of unexpected phenomena, if compared with the euclidean case, see e.g, the introductions of \cite{elvise-14} and \cite{bonforte-13}. Another reason for the interest in this kind of research is the fact that some classical tools in the Euclidean case do not work in non-euclidean frameworks, therefore the analysis requires new ideas and alternative approaches that could be useful also in other contexts. As a matter of example, an analogous of the above-mentioned Emden transformation seems not known in non-euclidean settings, therefore different arguments, such as the exploitation of ad hoc Lyapunov functionals, fine asymptotic analysis and blow-up methods, must be exploited, see e.g., Sections  \ref{3}-\ref{5} below.

Coming back to \eqref{main_eq}, its investigation in non-euclidean settings turns out to be the natural subsequent step in order to complete the scenario of results for this equation; this motivates the present paper. More precisely, we provide a complete classification with respect to asymptotic behaviour, stability and intersections properties of radial smooth solutions to \eqref{main_eq} for manifolds $M$ having strictly negative sectional curvatures, possibly unbounded, see assumptions \eqref{A1}-\eqref{A3} in Section \ref{2}. It is worth mentioning that, even if in our proofs we take advantage of some arguments already employed in the study of the Lane-Emden-Fowler equation on Riemannian models, the fact of dealing with an exponential nonlinearity brought a number of considerable technical difficulties related, for instance, to the different sign and decay behaviour of solutions in the two cases. Moreover, our analysis made it possible to highlight a number of properties of solutions that cannot be observed in the flat case, see e.g., Remarks \ref{R1}, \ref{R2} and \ref{R3} in Section \ref{2}.

\medspace

\medspace

The paper is organised as follows. In Section \ref{2} we give the precise formulation of the problem and we state our main results about continuation and asymptotic behaviour of solutions (Proposition \ref{first_facts} and Theorem \ref{final}), stability properties (Theorems \ref{stability statement} and \ref{t:stability} ), and intersection properties (Theorem \ref{t:stab-inter} and Theorem \ref{intersection}); the remaining sections of the paper are devoted to the proofs. More precisely, in Section \ref{3} we prove Proposition \ref{first_facts} while Section \ref{4} is devoted to the proof of Theorem \ref{final}. Section \ref{5} contains a number of technical lemmas that will be exploited to prove Theorems \ref{t:stab-inter}, \ref{stability statement}, \ref{t:stability} and \ref{intersection}. At last, for the sake of the reader, in the Appendix, we briefly recall some well-known facts in the Euclidean case that highlight the role of the dimension $N$ in the stability and intersection properties analysis.

%%%%%%%%%%%%%%%%%%%%%%%%%%%%%%%%%%%%%%%%%%%%%%%%%%%%%%%%%%%%%%%%%%%%%%%%%%%%%%%%%%%%%%%%%%%%%%%%%

%%%%%%%%%%%%%%%%%%%%%%%%%%%%%%%%%%%%%%%%%%%%%%%%%%%%%%%%%%%
\section{Statement of the problem and main results}\label{2}

\subsection{Notations}
	Let $\psi$ be the function introduced in \eqref{metric}. We assume that $\psi$ satisfies
 \begin{align}\label{A1}
	\tag{A1}
   \text{$\psi\in C^2([0,\infty))$, \ $\psi>0$ in $(0,+\infty)$, \ $\psi(0)=\psi^{\prime\prime}(0)=0$ \ and \ $\psi^\prime(0)=1$;}
\end{align}
and
\begin{align}\label{A2}
	\tag{A2}
    \text{$\psi^\prime(r)>0$ \quad for any $r>0$ \,.}
\end{align}

We recall that the Riemannian model associated with the choice $\psi(r)=\sinh r$ is a well-known representation of the hyperbolic space $\hn$, see e.g., \cite{Gr} and the references therein, while the Euclidean space $\rn$ corresponds to $\psi(r)= r$.

The following list summarizes some notations we shall use throughout this paper.

\begin{itemize}
\item[-] For any $P\in M$ we denote by $T_P M$ the tangent space to $M$ at the point $P$.
\item[-] For any $P\in M$ and $U_1,U_2\in T_P M$ we denote by $\langle U_1,U_2\rangle_g$ the scalar product on $T_P M$ associated with the metric $g$.
\item[-] For any $P\in M$ and $U\in T_P M$ we denote by $|U|_g:=\sqrt{\langle U,U\rangle_g}$ the norm of the vector $U$.
\item[-] ${\rm d}V_g$ denotes the volume measure in $(M,g)$.
\item[-] $\nabla_g$ denotes the Riemannian gradient in $(M,g)$ and in spherical coordinates it is given by
$\nabla_g u(r,\omega)=\frac{\partial u}{\partial r}(r,\omega) \, \frac{\partial}{\partial r}+\frac{r}{(\psi(r))^2}  \, \nabla_\omega u(r,\omega)$ for any $u\in C^1(M)$, where we denoted by $\nabla_\omega$ the Riemannian gradient in the unit sphere $\mathbb S^{N-1}$.
\item[-] $\Delta_g$ denotes the Laplace-Beltrami operator in $(M,g)$ and in spherical coordinates it is given by
$\Delta_g u(r,\omega)=\frac{\partial^2 u}{\partial r^2}(r,\omega)+(N-1)\frac{\partial u}{\partial r}(r,\omega)+\frac{1}{(\psi(r))^2} \, \Delta_{\omega} u(r,\omega)$ for any $u\in C^2(M)$, where we denoted by $\Delta_\omega$ the Laplace-Beltrami operator in the unit sphere $\mathbb S^{N-1}$.
\item[-] $C^\infty_c(M)$ denotes the space of $C^\infty(M)$ functions compactly supported in $M$.
\end{itemize}

\bigskip

From the above notations, we deduce that if $u\in C^2(M)$ is a radial function, then
\begin{align}\label{relation_model}	
    \Delta_{g}u(r)=u^{\prime\prime}(r)+(N-1)\frac{\psi^\prime(r)}{\psi(r)}\, u^\prime(r)=\frac{1}{(\psi(r))^{N-1}}[(\psi(r))^{N-1}u^\prime(r)]^\prime.
\end{align}

Since our aim is to study \emph{smooth} radial solutions to \eqref{main_eq}, for any $\alpha \in \R$, we focus our attention on the following initial value problem
\begin{equation}\label{regular_ode_model}
	\begin{dcases}
		-u^{\prime\prime}(r)-(N-1)\frac{\psi^\prime(r)}{\psi(r)}u^\prime(r)=e^{u(r)} & (r>0) \\
		u(0)=\alpha  & \\
		u^\prime(0)=0. & \\
	\end{dcases}
\end{equation}
The existence and uniqueness of a local solution $u(r)$ to \eqref{regular_ode_model} in $0 \leq r <  R$ (here $R$ denotes the maximal interval of existence) follows by arguing as in Proposition 1 in the Appendix of \cite{serrin}. A classical argument allows to prove that actually $R=+\infty$, thus showing that the solution $u=u(r)$ to \eqref{regular_ode_model} is globally define in $[0,+\infty)$:

\begin{proposition}\label{first_facts}
	Let $N\geq 2$. Suppose that $\psi$ satisfies
assumptions \eqref{A1}-\eqref{A2}. For any $\alpha\in \R$ the local solution to \eqref{regular_ode_model} may be continued to the whole interval $[0,+\infty)$. Moreover, the functions $r\mapsto u^\prime(r)$ and $r\mapsto e^{u(r)}$ are bounded in $[0,+\infty)$, $u'(r)<0$ for any $r>0$ and in particular $u$ is decreasing in $[0,+\infty)$.
\end{proposition}

\subsection{Asymptotic behaviour}
In order to study the asymptotic behaviour of global solutions to \eqref{regular_ode_model} we require the additional condition:
\begin{align}\label{A3}
	\tag{A3}
	\text{$\lim_{r\rightarrow +\infty}\frac{\psi^\prime(r)}{\psi(r)}=:\Lambda \in (0,\infty].$}
	\end{align}

Clearly, the hyperbolic space
satisfies condition \eqref{A3} and Riemannian models which are
asymptotically hyperbolic satisfy it as well. Furthermore, such a condition allows for \it unbounded \rm
negative sectional curvatures: a typical example in which this can
hold corresponds to the choice $\psi(r)=e^{r^a}$ for a given $a>1$
and $r$ large, a case for which (see \cite[Section 1.1]{elvise-14}) sectional curvatures in the radial direction diverge like $-a^2r^{2(a-1)}$ as $r\to+\infty$. In addition, we remark that under assumptions \eqref{A1}-\eqref{A3}, the $L^2$ spectrum of $-\Delta_g$ is
bounded away from zero whereas if $\lim_{r\rightarrow +\infty}\frac{\psi^\prime(r)}{\psi(r)}=0$ then there is no gap in the
$L^2$ spectrum of $-\Delta_g$, see e.g., \cite[Lemma 4.1]{elvise-14}. Moreover, it can be proved that if the radial
sectional curvature goes to zero as $r\to+\infty$ then necessarily $\lim_{r\rightarrow +\infty}\frac{\psi^\prime(r)}{\psi(r)}=0$, see again \cite[Lemma 4.1]{elvise-14}, therefore no spectral gap is present and the expected
picture is of Euclidean type.

In the following statement we show that the asymptotic behaviour of solutions of \eqref{regular_ode_model} is related to the behaviour at infinity of the ratio $\frac{\psi^\prime(r)}{\psi(r)}$ and hence, for what remarked above, to the curvatures of the manifold:

\begin{theorem}\label{final} Let $N\geq 2$. Suppose that $\psi$ satisfies
assumptions \eqref{A1}-\eqref{A3}. Finally, in the
case $\Lambda=+\infty$ we also assume that
 \begin{align}\label{A6}
	\tag{A4}	\left[\log\left(\frac{\psi^\prime(r)}{\psi(r)}\right)\right]^\prime = O(1) \quad \text{as } r\rightarrow +\infty\,.
\end{align}
 Let $u$ be a solution to \eqref{regular_ode_model}. Then two cases may occur:
\begin{itemize}
\item[$(i)$] if $\frac{\psi}{\psi'}\in L^1(0,\infty)$, then
$$
\lim_{r\rightarrow +\infty} u(r)\in (-\infty, \alpha) \, ;
$$
\item[$(ii)$] if $\frac{\psi}{\psi'}\not\in L^1(0,\infty)$, then $u$ goes to $-\infty$ with the following rate:
\begin{align*}
		\lim_{r\rightarrow +\infty}\frac{u(r)}{\log\big(\int_{0}^{r}\frac{\psi(s)}{\psi^\prime(s)}{\rm d}s\big)}=-1 \, ;
	\end{align*}
in particular, when $\Lambda\in (0,+\infty)$ we have
\begin{align*}
		\lim_{r\rightarrow +\infty}\frac{u(r)}{\log r}=-1.
	\end{align*}
\end{itemize}
\end{theorem}

\medskip

\begin{remark}\label{R1}
{\rm
As a prototype of function $\psi$ satisfying the assumptions of Theorem \ref{final} when $\Lambda= + \infty$ consider the function $\psi(r)=re^{r^{2\gamma}}$ with $\gamma > \frac{1}{2}$. If $\frac{1}{2}< \gamma \leq 1$ case $(ii)$ occurs while if $\gamma>1$ assumption $(i)$ holds. Clearly, if $M=\hn$ we have that $\Lambda=1$ and case $(ii)$ occurs. From \cite{joseph-73} we recall that in the flat case solutions diverge to $-\infty$ being asymptotically equivalent to $-2\log r$, therefore the effect of curvatures, in general, results in a slower decay of solutions, indeed they may even remain bounded if case $(i)$ occurs.}
\end{remark}

 \subsection{Stability results and intersection properties of solutions}

 Let us start with the definition of stability.

\begin{definition}
A solution  $u \in C^2(M)$ to \eqref{main_eq} is stable if
\begin{align}\label{stable_def}
	\int_M |\gradg v|_g^2\, \dvg-\int_M e^u \,  v^2\,\dvg\geq 0 \quad \forall v\in C_c^{\infty}(M).
\end{align}
If $u$ does not satisfy \eqref{stable_def}, we say that it is unstable.
\end{definition}

It is well known that stability plays an important role in the classification of solutions of elliptic partial differential equations and the analysis of qualitative properties of solutions, see e.g., the seminal paper \cite{bv}.  In this section, we provide a complete classification of smooth radial solutions of \eqref{main_eq} with respect to stability and we show that the same conditions determine either the stability or the intersection properties. The relationship between stability properties and intersection properties is clarified by the following result.

\begin{theorem} \label{t:stab-inter}
   Let $N\ge 2$ and let $\psi$ satisfy \eqref{A1}-\eqref{A3}. Then the following statements hold:

   \begin{itemize}
     \item[$(i)$] let $\alpha$ and $\beta$ be two distinct real numbers and let $u_\alpha$ and $u_\beta$ be the corresponding solutions of \eqref{regular_ode_model}, i.e. $u_\alpha(0) = \alpha$ and $u_\beta(0) = \beta$. If $u_\alpha$ and $u_\beta$ are stable then they do not intersect;

     \item[$(ii)$] if $u_\alpha$ is a unstable solution of \eqref{regular_ode_model} for some $\alpha \in \R$, then for any $\beta>\alpha$ we have that $u_\beta$ intersects $u_\alpha$ at least once.
   \end{itemize}
\end{theorem}

We observe that in the proof Theorem \ref{t:stab-inter} -$(ii)$, we actually show the validity of a more general result involving also non-radial smooth solutions of \eqref{main_eq}. Indeed, in Lemma \ref{cpoli} we prove that if $u$ is a smooth unstable solution of \eqref{main_eq} then \eqref{main_eq} does not admit any smooth solution $v$ satisfying $v>u$ in $M$.

We now state the two main results about the stability of radial smooth solutions of \eqref{main_eq} characterized by dimensions
$2\le N\le 9$ and $N \ge 10$, respectively.

\begin{theorem} \label{stability statement}
Let $2\leq N\leq 9$ and let $\psi$ satisfy \eqref{A1}-\eqref{A3}. For any $\alpha \in \R$ denote by $u_{\alpha}$ the unique solution to \eqref{regular_ode_model}. Then there exist $\eta\in \R$ such that
\begin{itemize}
	\item[$(i)$] if $\alpha\in (-\infty, \eta]$ then $u_\alpha$ is stable;
	\item[$(ii)$] if $\alpha>\eta$ then $u_\alpha$ is unstable.
\end{itemize}
Furthermore, we have that $\eta\geq \log(\lambda_1(M))$ with the strict inequality if $\frac{\psi}{\psi^\prime}\notin L^1(0,+\infty)$ where $\lambda_1(M)$ denotes the bottom of the spectrum of $-\Delta_g$ in $M$.
\end{theorem}

\begin{theorem} \label{t:stability} Let $N \ge 10$ and let $\psi$ satisfy  \eqref{A1}-\eqref{A3} and the additional condition
    \begin{equation} \label{eq:A4}
     \tag{A5} \psi\in C^3([0,+\infty)), \quad  [\log(\psi'(r))]''>0  \quad \text{for } r>0 \, .
   \end{equation}
 Then all solutions to \eqref{regular_ode_model} are stable.
\end{theorem}

\begin{remark}\label{R2}
{\rm
The fact that $N=10$ is a critical threshold for stability is well-known in the Euclidean case, see \cite{farina2} and the Appendix. Here a new critical value $\eta$ arises which has no analogous in the flat case where solutions are always unstable if $2\le N\le 9$ and Theorem \ref{stability statement}-$(i)$ never occurs. In other words, we can say that, by the effect of assumption \eqref{A3}, the critical dimension does not exist if $\alpha \leq \eta$ since solutions are stable for all $N\geq 2$. \\
About assumption \eqref{eq:A4}, it is technical but it includes the most interesting examples. Indeed, it holds in the relevant case of the hyperbolic space, since $[\log(\psi'(r))]''=(\cosh r)^{-2}$. Models with unbounded curvatures satisfying \eqref{eq:A4} can be build as well, e.g., by taking $\psi(r)=e^{r^a}$ with $a>1$ for $r\geq 1$, indeed $[\log(\psi'(r))]''=\frac{a-1}{r^2}(ar^a-1)>0$ for $r\geq 1$.}
\end{remark}

 At last, concerning intersection properties, we have:

 \begin{theorem}\label{intersection}
Let $N\geq 2$, let $\psi$ satisfy \eqref{A1}-\eqref{A3} and let $\eta$ be as in the statement of Theorem \ref{stability statement}. We have that:
\begin{itemize}
\item[$(i)$] if $2\leq N\leq 9$ and $u_\alpha$, $u_\beta$ are two solutions of \eqref{regular_ode_model} with $\alpha, \beta \in (-\infty, \eta]$ then $u_\alpha$ and $u_\beta$ do not intersect;
\item[$(ii)$] if $2\leq N\leq 9$ and $u_\alpha$, $u_\beta$ are two solutions of \eqref{regular_ode_model} with $\alpha, \beta>\eta$ then $u_\alpha$ and $u_\beta$ intersect at least once;
\item[$(iii)$] if $N\geq 10$ and the additional condition \eqref{eq:A4} holds true, then for any $\alpha, \beta \in \R$ the corresponding solutions $u_\alpha$ and $u_\beta$ of \eqref{regular_ode_model} do not intersect.
\end{itemize}

\end{theorem}
\begin{remark}\label{R3}
{\rm
In $\R^N$ it is known, see e.g., \cite[Theorem 1.1]{tello-06}, that every two smooth radial solutions intersect each other once if $N=2$ and infinitely many times if $3\leq N\leq 9$ while for $N\geq 10$ smooth radial solutions do not intersect. Therefore, the behaviour stated in Theorem \ref{intersection}-$(i)$ (with its counterpart Theorem \ref{stability statement}-$(i)$) is, as a matter of fact, an effect of non-vanishing curvatures.
}
\end{remark}

\section{Proof of Proposition \ref{first_facts}}\label{3}

	Let $(0,R)$ be the maximum interval where the local solution $u$ of \eqref{regular_ode_model} is defined and consider the Lyapunov functional
	\begin{align}\label{lyapunov_functional_model}
		F(r):=\frac{1}{2} (u^\prime(r))^2	+ e^{u(r)} \quad \text{for any } r\in (0, R) \, .
	\end{align}
	It is readily seen that
	\begin{align*}
		F^\prime(r)=u^\prime(r)[u^{\prime\prime}(r)+e^{u(r)}]=-(N-1)u^\prime(r)^2\frac{\psi^\prime(r)}{\psi(r)}\leq 0 \quad \text{for any } r\in (0, R) \, .
	\end{align*}
	This proves that $F$ is decreasing and $0< F(r)\leq e^\alpha$ for any $r\in (0, R)$, therefore both $u^\prime(r)$ and $e^{u(r)}$ are bounded in $(0,R)$. On the other hand, integrating \eqref{relation_model} from $0$ to $r$ we deduce
	\begin{align*}
		(\psi(r))^{N-1}u^\prime(r)=-\int_{0}^{r}(\psi(s))^{N-1} \, e^{u(s)}\ds< 0 \quad \text{for any } r \in (0,R) \, .
	\end{align*}
	Therefore, $u^\prime$ is a negative function in $(0,R)$ and hence $u$ is decreasing in $(0,R)$ as claimed.
	
	Finally, we show that $R=+\infty$.  Assume by contradiction that $R$ is finite and let
	\begin{align*}
		l:=\lim_{r\rightarrow R^-}u(r) \in [-\infty,\alpha) \, .
	\end{align*}
	If $l=-\infty$ we immediately get a contradiction since
	\begin{align*}
		u(r)=\alpha +\int_{0}^{r}u^\prime(s)\ds \qquad \text{for any } r\in (0,R)
	\end{align*}
	and taking the limit as $r\rightarrow R^-$ we have that the left hand side goes to $-\infty$ while the right hand side  remains bounded since $u^\prime(r)$ is bounded. Being $l$ finite, from local existence for solutions of ordinary differential equations with initial value conditions, the solution $u$ may be extended in a right neighbourhood of $R$ thus contradicting the maximality of $R$. This shows that $R=+\infty$ completes the proof of the proposition.
\endproof

\section{Proof of Theorem \ref{final} }\label{4}

Throughout this section, we will always assume the validity of conditions \eqref{A1}-\eqref{A3} and when $\Lambda=+\infty$ in \eqref{A3} we also assume the validity of the additional condition \eqref{A6}. We start by showing that $u'$ admits a limit at infinity and this limit is zero:

\begin{lemma} \label{l:u'-0}
  Let $u$ be the unique solution of \eqref{regular_ode_model}. Then
   \begin{equation*}
      \lim_{r \to +\infty} u'(r)=0 \, .
   \end{equation*}
\end{lemma}

\begin{proof}
    From Proposition \ref{first_facts} we know that $u$ is decreasing in $[0,+\infty)$ and hence it admits a limit as $r\to +\infty$. We put
    \begin{equation} \label{eq:l}
        l:=\lim_{r\to +\infty} u(r)\in [-\infty,\alpha) \, ,
    \end{equation}
    where $\alpha=u(0)$.
    Let $F$ be the function defined in \eqref{lyapunov_functional_model}. Since $F$ is non-increasing, then $0< F(r)\leq F(0)=e^\alpha$ for any $r \ge 0$. Hence, there exists $c \in [0,+\infty)$ such that $c=\lim_{r\rightarrow +\infty}F(r)$.

    By \eqref{eq:l} we deduce that $\bar l:=\lim_{r\to +\infty} e^{u(r)}\in [0,+\infty)$ so that by \eqref{lyapunov_functional_model} we obtain
    \begin{equation} \label{eq:limit-u'}
       \lim_{r \to +\infty} u'(r)=\lim_{r\to +\infty} -\sqrt{2F(r)-e^{u(r)}}=-\sqrt{2c-\bar l}=\gamma \in (-\infty,0] \, .
    \end{equation}
    It remains to prove that $\gamma=0$. Suppose by contradiction that $\gamma<0$.
    This implies
    \begin{equation} \label{eq:lim-u}
        \lim_{r\to +\infty} u(r)=\lim_{r \to +\infty} \left[\alpha+\int_0^r u'(s) \, {\rm d}s\right]=-\infty \, .
    \end{equation}
    Letting $\Lambda\in (0,+\infty]$ be as in \eqref{A3}, by \eqref{regular_ode_model}, \eqref{eq:limit-u'} and \eqref{eq:lim-u} we obtain
    \begin{equation} \label{eq:claim-u''}
       \lim_{r\to +\infty} u''(r)=\lim_{r \to +\infty} \left[-(N-1) \frac{\psi'(r)}{\psi(r)} \, u'(r)-e^{u(r)}\right]=-(N-1) \Lambda \gamma \in (0,+\infty] \, .
    \end{equation}
    In turn, by \eqref{eq:claim-u''} and \eqref{regular_ode_model}, we infer that
    \begin{equation*}
    %\label{eq:lim-u'-bis}
        \lim_{r\to +\infty} u'(r)=\lim_{r \to +\infty} \int_0^r u''(s) \, {\rm d}s=+\infty \, ,
    \end{equation*}
    in contradiction with \eqref{eq:limit-u'}. This completes the proof of the lemma.
\end{proof}

By exploiting Lemma \ref{l:u'-0} we prove:

\begin{lemma}\label{limitexist}
	Let $u$ be the unique solution of \eqref{regular_ode_model}. Then the following limit exists
	\begin{align*}
		\lim_{r\rightarrow +\infty}\frac{u^{\prime\prime}(r)}{u^\prime(r)}\frac{\psi(r)}{\psi^\prime(r)} \,.
	\end{align*}
\end{lemma}

\begin{proof} We divide the proof into two parts depending on whether $\Lambda$ is finite or not.

{\bf The case $\Lambda \in (0,+\infty)$.} From \eqref{regular_ode_model}, we have that
	\begin{align}\label{et_eq_a_1_model}
		-\frac{u^{\prime\prime}(r)}{u^\prime(r)}-(N-1)\frac{\psi^\prime(r)}{\psi(r)}=\frac{e^{u(r)}}{u^\prime(r)}.
	\end{align}
	The above identity suggests that if $\chi(r):=\frac{e^{u(r)}}{u^\prime(r)}$ admits a limit as $r \to +\infty$, then the same does the function $\frac{u^{\prime\prime}(r)}{u^\prime(r)}$. This means that the proof of Lemma \ref{limitexist} follows once we prove that $\chi$ admits a limit as $r\to +\infty$.

We proceed by contradiction assuming that $\chi$ does not admit a limit as $r\to +\infty$.
% Clearly, Lemma \ref{blow_u_model} gives $\chi(r)<0,$ for  $r>0$.
By direct computation, we see that
	\begin{align} \label{eq:chi-prime}
		\chi^\prime(r)=\frac{e^{u(r)}[({u^\prime}(r))^2-u^{\prime\prime}(r)]}{({u^\prime}(r))^2}  \quad \text{for any } r>0.
	\end{align}

We may assume that $\chi$ admits infinitely many local maxima and minima at some points $r_m$, with $r_m\rightarrow+\infty$ as $m\rightarrow+\infty$, and $\chi(r_m)$ does not admit a limit as $m\to +\infty$.

In particular by \eqref{eq:chi-prime} we have $(u'(r_m))^2-u''(r_m)=0$ for any $m$. Hence, evaluating \eqref{et_eq_a_1_model} at $r_m$, we obtain
	\begin{align*}
		-u^\prime(r_m)-(N-1)\frac{\psi^\prime(r_m)}{\psi(r_m)}=\chi(r_m)\,,
	\end{align*}
	for any $m$. Now, by \eqref{A3} and Lemma \ref{l:u'-0}, we find that $\chi(r_m)\rightarrow -(N-1)\Lambda$ as $m\rightarrow +\infty$, a contradiction. This completes the proof of the lemma in the case $\Lambda \in (0,+\infty)$.

{\bf The case $\Lambda=+\infty$.} We proceed similarly to the previous case by writing \eqref{regular_ode_model} in the form 	\begin{align} \label{eqknown}
        -\frac{u^{\prime\prime}(r)}{u^\prime(r)}\frac{\psi(r)}{\psi^\prime(r)}-(N-1)
        =\frac{e^{u(r)}}{u^\prime(r)}\frac{\psi(r)}{\psi^\prime(r)}
	\end{align}
and by defining this time $\chi(r):=\frac{e^{u(r)}}{u^\prime(r)}\frac{\psi(r)}{\psi^\prime(r)}$ in such a way that if $\chi$ admits a limit as $r\rightarrow +\infty$ then the same conclusion occurs for $\frac{u^{\prime\prime}(r)}{u^\prime(r)}\frac{\psi(r)}{\psi^\prime(r)}$ and we complete the proof of the lemma also in this case.
By contradiction, assume that $\chi$ does not admit a limit as $r\to +\infty$. A simple computation gives
	\begin{align} \label{eq:chi'}
		\chi^\prime(r)=\chi(r)\left\{u^\prime(r)-\frac{u^{\prime\prime}(r)}{u^\prime(r)}
         -\left[\log\left(\frac{\psi^\prime(r)}{\psi(r)}\right)\right]^\prime\right\} \qquad \text{for any } r>0.
	\end{align}
We may assume that $\chi$ admits infinitely many local maxima and minima at some points $r_m$ with $r_m \rightarrow +\infty$ as $m\rightarrow+\infty$, and $\chi(r_m)$ does not admit a limit as $m\to +\infty$. In particular, by \eqref{eq:chi'}, we have that
\begin{align} \label{eq:id-00}
   & u^\prime(r_m)-\frac{u^{\prime\prime}(r_m)}{u^\prime(r_m)}
         -\left[\log\left(\frac{\psi^\prime(r_m)}{\psi(r_m)}\right)\right]^\prime =0 \, .
\end{align}
Hence, evaluating \eqref{eqknown} at $r_m$ and using \eqref{eq:id-00}, we obtain
	\begin{align*}
		-(N-1)-\frac{\psi(r_m)u^\prime(r_m)}{\psi^\prime(r_m)}+\frac{\psi(r_m)}{\psi^\prime(r_m)}\left[\log
         \left(\frac{\psi^\prime(r_m)}{\psi(r_m)}\right)\right]^\prime=\chi(r_m) \, ,
	\end{align*}
for any $m$.

Finally, by Lemma \ref{l:u'-0}, \eqref{A3} and \eqref{A6}, we find that $\chi(r_m)\rightarrow -(N-1)$ for $m\rightarrow +\infty$, a contradiction. The proof of the lemma is complete also in this case.
\end{proof}

Our next purpose is to show that the limit in Lemma \ref{limitexist} must be $0$ under the additional assumption:
\begin{equation}\label{eq:lim-u-infty}
   \lim_{r \to +\infty} u(r)=-\infty \, .
\end{equation}
In Lemma \ref{l:integrable} below we discuss the occurrence of \eqref{eq:lim-u-infty} and we provide a sufficient condition for \eqref{eq:lim-u-infty} in terms of the integrability properties of the ratio $\psi/\psi'$.

Before proving in Lemma \ref{limitinfty} that the limit in Lemma \ref{limitexist} is zero when \eqref{eq:lim-u-infty} holds, we state the following result which deals with the behaviour of $\psi$ at infinity.

\begin{lemma}
%\label{l:Lambda-infty}
Let $\Lambda$ be as in \eqref{A3}, the following statements hold true:

\begin{itemize}
  \item[$(i)$] for any $\Lambda \in (0,+\infty]$ we have that
  \begin{align} \label{A5}
	    \lim_{r\rightarrow +\infty} \psi(r)=+\infty
	\end{align}
 and moreover for any $M>0$ and $0<\delta<N-1$ we have that
     \begin{align}\label{A7}
	    \psi^{-(N-1)+\delta}\in L^1(M,\infty) \, ;
    \end{align}

   \item[$(ii)$] if $\Lambda\in (0,+\infty)$ then there exist $C>0$ and $M>0$ such that
       \begin{equation}\label{eq:est-basso}
          \psi(r)>C e^{\frac \Lambda 2 \, r} \qquad \text{for any } r>M \, .
       \end{equation}
\end{itemize}

\end{lemma}

\begin{proof} Let us start with the proof of $(i)$. First of all, we observe that \eqref{A3} also reads $\lim_{r\rightarrow +\infty}\left[\log(\psi(r))\right]'=\Lambda>0$ and \eqref{A5} follows by writing
$$
  \log(\psi(r))=\int_R^r \left[\log(\psi(s))\right]'\, ds+\log(\psi(R))
$$
for some $R>0$ and letting $r \to +\infty$.

Let us proceed with the proof of \eqref{A7} when $\Lambda= + \infty$. By \eqref{A3} it follows that there exists $R>0$ such that $\psi'(r)> M \psi(r)$ for any $r>R$, for some positive constant $M>0$. This implies
\begin{align*}
    & \int_R^{+\infty} \frac{1}{(\psi(r))^{N-1-\delta}} \, {\rm d}r
     =\int_{\psi(R)}^{+\infty} \frac{1}{s^{N-1-\delta} \, \psi'(\psi^{-1}(s))} \, {\rm d}s  \\[10pt]
   & \qquad   <\frac 1M \int_{\psi(R)}^{+\infty} \frac{1}{s^{N-1-\delta}\, \psi(\psi^{-1}(s))} \, {\rm d}s
    =\frac 1M \int_{\psi(R)}^{+\infty} \frac{1}{s^{N-\delta} }\, {\rm d}s< +\infty \, .
\end{align*}
This proves \eqref{A7} and completes the proof of $(i)$ when $\Lambda= + \infty$. The validity of \eqref{A7} when $\Lambda< + \infty$ is an easy consequence of statement $(ii)$.

Let us proceed with the proof of $(ii)$. By \eqref{A3} we have that for any $\varepsilon>0$ there exists $r_\varepsilon>0$ such that
\begin{equation*}
  \Lambda-\varepsilon<[\log(\psi(r))]'<\Lambda+\varepsilon \qquad \text{for any } r>r_\varepsilon \, .
\end{equation*}
After integration we get
\begin{equation*}
   \psi(r_\varepsilon) e^{(\Lambda-\varepsilon)(r-r_\varepsilon)}<\psi(r)<\psi(r_\varepsilon) e^{(\Lambda+\varepsilon)(r-r_\varepsilon)} \qquad \text{for any } r>r_\varepsilon \, .
\end{equation*}
The proof of \eqref{eq:est-basso} now follows choosing $\varepsilon=\frac \Lambda 2$ and $C=\psi(r_\varepsilon) e^{-\frac{\Lambda}{2} \, r_\varepsilon}$ in the left inequality above.
\end{proof}

\medskip
\begin{lemma}\label{limitinfty}
	Let $u$ be the unique solution of \eqref{regular_ode_model} and suppose that $u$ satisfies \eqref{eq:lim-u-infty}. Then we have
	\begin{align*}
		\lim_{r\rightarrow +\infty}\frac{u^{\prime\prime}(r)}{u^\prime(r)}\frac{\psi(r)}{\psi^\prime(r)} =0\,.
	\end{align*}
\end{lemma}
\begin{proof}
	If $u^{\prime\prime}$ vanishes infinitely many times at infinity, then by Lemma \ref{limitexist} we are done. If this does not occur, using again Lemma \ref{limitexist} and, recalling Proposition \ref{first_facts} and assumptions \eqref{A1}-\eqref{A2}, we infer that the following limit exists
$$
   \lim_{r\rightarrow +\infty}\frac{u^\prime(r)\psi^\prime(r)}{u^{\prime\prime}(r)\psi(r)}=:L\in [-\infty,+\infty] \, .
$$
The proof of the lemma now follows if we prove that $|L|=+\infty$.

We divide the remaining part of the proof into two steps.

\medskip
	
{\bf Step 1.} We show that $L\neq 0$. By contradiction, assume $L=0$. Notice that by Lemma \ref{l:u'-0}, \eqref{eq:lim-u-infty}, de l'H\^opital's rule and \eqref{A3}
	\begin{align*}
	\lim_{r\rightarrow +\infty}\frac{e^{u(r)}}{u^{\prime}(r)}=\lim_{r\rightarrow +\infty}\frac{u^{\prime}(r) e^{u(r)}}{u^{\prime\prime}(r)} =\lim_{r\rightarrow +\infty}\frac{u^{\prime}(r)\psi'(r) }{u^{\prime\prime}(r) \psi(r)}\, e^{u(r)}\, \frac{\psi(r)}{\psi'(r)} =0
	\end{align*}
by which, from \eqref{A3}, we readily get
	\begin{align}\label{T_et_1_model_infty}
	   \lim_{r\rightarrow +\infty}\frac{e^{u(r)}\psi(r)}{u^{\prime}(r)\psi^\prime(r)}=0 \, .
	\end{align}
Recalling \eqref{eqknown}, by \eqref{T_et_1_model_infty} we deduce that
	\begin{align*}
		\lim_{r\rightarrow +\infty}\frac{u^{\prime\prime}(r)}{u^\prime(r)}\frac{\psi(r)}{\psi^\prime(r)}=-(N-1).
	\end{align*}
	This yields $\lim_{r\rightarrow +\infty}\frac{u^\prime(r)\psi^\prime(r)}{u^{\prime\prime}(r)\psi(r)}=-\frac{1}{N-1}$, a contradiction.

This completes the proof of Step 1.

\medskip
	
	{\bf Step 2.} We now prove that $L$ cannot be finite. Thanks to Step 1, we may assume by contradiction that $L\in \mathbb{R}\setminus\{0\}$. Arguing as in Step 1 we get that \eqref{T_et_1_model_infty} also holds in this case and, in turn, we obtain that $L=-\frac{1}{N-1}$. Therefore, for any $\varepsilon>0$ there exists $r_\varepsilon>$ such that for any $r>r_\varepsilon$
	$$
       -(N-1)-\varepsilon\leq \frac{u^{\prime\prime}(r)\psi(r)}{u^\prime(r)\psi^\prime(r)} \leq -(N-1)+\varepsilon \, ,
    $$
whence
    $$ [-(N-1)-\varepsilon]\, [\log(\psi(r))]^\prime\leq [\log (-u^\prime(r))]^\prime\leq
       [-(N-1)+\varepsilon]\, [\log(\psi(r))]^\prime \, .
    $$
Integrating from $r_\varepsilon$ to $r$, with $r>r_\varepsilon$, we deduce that
$$
   \log \bigg(\frac{-u^\prime(r)}{-u^\prime(r_\varepsilon)}\bigg)\leq [-(N-1)+\varepsilon]\log\bigg(\frac{\psi(r)}{\psi(r_\varepsilon)}\bigg)$$
and, in turn, that
	$$ -u^\prime(r)\leq A_\varepsilon (\psi(r))^{-(N-1)+\varepsilon}\,,
$$
where $A_\varepsilon=\frac{-u^\prime(r_\varepsilon)}{(\psi(r_\varepsilon))^{-(N-1)+\varepsilon}}$ is a positive constant. By a further integration, for any $r>r_\varepsilon$, we obtain
\begin{align} \label{eq:up-est}
	-u(r)+u(r_\varepsilon)\leq A_\varepsilon \int_{r_\varepsilon}^{r}(\psi(s))^{-(N-1)+\varepsilon}\,{\rm d}s \, .
\end{align}
By \eqref{A7} we deduce that in both cases $\Lambda=+\infty$ and $\Lambda \in (0,+\infty)$, the function
$$
  (\psi(s))^{-(N-1)+\varepsilon} \in L^1(r_\varepsilon,+\infty)
$$
provided that $\varepsilon<N-1$. This means that the right-hand side of \eqref{eq:up-est} admits a finite limit as $r \to +\infty$ and this is absurd since the left-hand side of \eqref{eq:up-est} blows up to $+\infty$ in view of \eqref{eq:lim-u-infty}.
\end{proof}

Finally, we determine the exact asymptotic behaviour at infinity for solutions of \eqref{regular_ode_model} satisfying \eqref{eq:lim-u-infty}.

\begin{lemma}\label{sol_asymptotic_model_infty}
	Let $u$ be the unique solution of \eqref{regular_ode_model} and suppose that $u$ satisfies \eqref{eq:lim-u-infty}. Then
	\begin{align}\label{decay_est_model_infty}
		\lim_{r\rightarrow +\infty}\frac{e^{-u(r)}}{\int_{0}^{r}\frac{\psi(s)}{\psi^\prime(s)}{\rm d}s}=\frac{1}{N-1}.
	\end{align}
	Consequently, the following decay estimate holds
	\begin{align}\label{rate_decay_est_model_infty}
		\lim_{r\rightarrow +\infty}\frac{u(r)}{\log\big(\int_{0}^{r}\frac{\psi(s)}{\psi^\prime(s)}{\rm d}s\big)}=-1.
	\end{align}
In particular if $\Lambda \in (0,+\infty)$ then
   \begin{equation} \label{eq:rate-L-finite}
       \lim_{r\rightarrow +\infty}\frac{u(r)}{\log r}=-1
   \end{equation}
\end{lemma}

\begin{proof}
	By \eqref{regular_ode_model}, Lemma \ref{limitexist} and Lemma \ref{limitinfty}, we get
	\begin{align*}
		 \lim_{r\rightarrow +\infty} -\frac{e^{u(r)}\psi(r)}{u^\prime(r)\psi^\prime(r)}=N-1 \, .
	\end{align*}
	Therefore, for any $\varepsilon>0$ there exists $r_\varepsilon>0$ such that
	$$
        \bigg(\frac{1}{N-1}-\varepsilon\bigg)\frac{\psi(r)}{\psi^\prime(r)}<\big(e^{-u(r)}\big)^\prime<
        \bigg(\frac{1}{N-1}+\varepsilon\bigg)\frac{\psi(r)}{\psi^\prime(r)} \, .
    $$
Integrating from $r_\varepsilon$ to $r$, for any $r>r_\varepsilon$ we get
\begin{align} \label{eq:est-psi-psi'}
	e^{-u(r_\varepsilon)}+\bigg(\frac{1}{N-1}-\varepsilon\bigg)\int_{r_\varepsilon}^r\frac{\psi(s)}{\psi^\prime(s)}\,
    {\rm d}s<e^{-u(r)}<e^{-u(r_\varepsilon)}+\bigg(\frac{1}{N-1}+\varepsilon\bigg)
    \int_{r_\varepsilon}^r\frac{\psi(s)}{\psi^\prime(s)}\,{\rm d}s.
\end{align}
Now, if $\frac{\psi}{\psi^\prime}$ is integrable in a neighbourhood of infinity, we reach a contradiction with \eqref{eq:lim-u-infty}, therefore $\frac{\psi}{\psi^\prime}$ has to be not integrable in a neighbourhood of infinity and we obtain
\begin{align*}
	\lim_{r\rightarrow +\infty}\frac{e^{-u(r)}}{\int_{r_\varepsilon}^r\frac{\psi(s)}{\psi^\prime(s)}\,{\rm d}s}=\frac{1}{N-1}.
\end{align*}
Then, \eqref{decay_est_model_infty} readily follows from the above limit by recalling \eqref{A1}. The limit \eqref{rate_decay_est_model_infty} follows from \eqref{eq:est-psi-psi'} with a similar argument.

It remains to prove \eqref{eq:rate-L-finite} when $\Lambda \in (0,+\infty)$. We proceed by considering the limit:
\begin{align*}
   & \lim_{r \to +\infty}  \frac{\log\left(\int_0^r \frac{\psi(s)}{\psi'(s)} \, {\rm d}s\right)}{\log r}
   =\lim_{r \to +\infty} \frac{\psi(r)}{\psi'(r)} \frac{r}{\int_0^r \frac{\psi(s)}{\psi'(s)} \, {\rm d}s}
   =\frac{1}{\Lambda} \, \lim_{r \to +\infty} \frac{r}{\int_0^r \frac{\psi(s)}{\psi'(s)} \, {\rm d}s} \\[10pt]
   & \qquad =\frac{1}{\Lambda} \, \lim_{r \to +\infty} \frac{\psi'(r)}{\psi(r)}=\frac{1}{\Lambda} \cdot \Lambda=1 \, ,
\end{align*}
where we used twice de l'H\^opital rule.

This proves that $\log\left(\int_0^r \frac{\psi(s)}{\psi'(s)} \, {\rm d}s\right)\sim \log r$ as $r \to +\infty$ and the proof of \eqref{eq:rate-L-finite} follows from \eqref{rate_decay_est_model_infty}.

\end{proof}

At last, we provide a sufficient condition for \eqref{eq:lim-u-infty} in terms of integrability properties of the ratio $\psi/\psi'$.

\begin{lemma} \label{l:integrable}
  Let $u$ be the unique solution of \eqref{regular_ode_model}. Then the two alternatives hold:

  \begin{itemize}
    \item[$(i)$] if $\frac{\psi}{\psi'} \in L^1(0,\infty)$ then
       \begin{equation*}
           \lim_{r \to +\infty} u(r)\in (-\infty,\alpha) \, ;
       \end{equation*}

     \item[$(ii)$] if $\frac{\psi}{\psi'} \not\in L^1(0,\infty)$ then
       \begin{equation*}
           \lim_{r \to +\infty} u(r)=-\infty \, .
       \end{equation*}
  \end{itemize}
\end{lemma}

\begin{proof} The existence of the limit of $u$ as $r\to +\infty$ is known from Proposition \ref{first_facts} as well as the fact that this limit is less than $\alpha$. It remains to prove that the limit is finite in case $(i)$ and $-\infty$ in case $(ii)$.

Let us start with the proof of $(i)$. Suppose by contradiction that the limit is $-\infty$ so that \eqref{eq:lim-u-infty} holds true. Then we can apply Lemma \ref{limitinfty} and proceed as in the proof of Lemma \ref{sol_asymptotic_model_infty} to obtain \eqref{eq:est-psi-psi'}. The integrability of $\psi/\psi'$ shows that $u$ remains bounded as $r\to +\infty$, a contradiction.

Let us proceed with the proof of $(ii)$. Set $l_1:=\lim_{r \to +\infty} u(r)$. Suppose by contradiction that $l_1$ is finite. We claim that
\begin{equation} \label{eq:claim-psi-u'}
   \lim_{r \to +\infty} \frac{\psi(r)}{u'(r) \psi'(r)}=-\infty \, .
\end{equation}

From the proof of Lemma \ref{limitexist} we know that the function $\frac{e^{u(r)}}{u'(r)} \frac{\psi(r)}{\psi'(r)}$ admits a limit as $r\to +\infty$, hence, since
$e^{-u(r)}\to e^{-l_1}\in (0,+\infty)$ as $r\to +\infty$, the limit in \eqref{eq:claim-psi-u'} exists and it belongs to $[-\infty,0]$, thanks to \eqref{A1}, \eqref{A2} and Proposition \ref{first_facts}.

Let us denote by $l_2\leq 0$ the limit in \eqref{eq:claim-psi-u'} and suppose by contradiction that it is finite.

Then for any $\varepsilon>0$ there exists $r_\varepsilon>0$ such that
\begin{equation*}
  l_2-\varepsilon <\frac{\psi(r)}{u'(r) \psi'(r)}<l_2+\varepsilon \qquad \text{for any } r>r_\varepsilon
\end{equation*}
and from the left inequality, it follows that
\begin{equation*}
   0<\frac{\psi(r)}{\psi'(r)}<(l_2-\varepsilon) u'(r) \qquad \text{for any } r>r_\varepsilon \, .
\end{equation*}

Integrating the inequality above we obtain
\begin{equation*}
  0<\int_{r_\varepsilon}^{r} \frac{\psi(s)}{\psi'(s)} \, {\rm d}s<(l_2-\varepsilon) [u(r)-u(r_\varepsilon)]
  \qquad \text{for any } r>r_\varepsilon  \, .
\end{equation*}
Passing to the limit as $r\to +\infty$ and recalling that $l_1$ is finite we infer that $\psi/\psi'$ is integrable at infinity, in contradiction with the assumption in $(ii)$. This completes the proof of \eqref{eq:claim-psi-u'}.

Combining \eqref{eq:claim-psi-u'} and \eqref{eqknown} with the fact that $l_1$ is finite we deduce that
\begin{align*}
   & \lim_{r\to +\infty} \frac{u''(r)}{u'(r)} \frac{\psi(r)}{\psi'(r)}
      =-(N-1)-\lim_{r \to +\infty} \frac{e^{u(r)}}{u'(r)} \frac{\psi(r)}{\psi'(r)}=+\infty \, .
\end{align*}
This implies that for any $M>0$ there exists $r_M>0$ such that
\begin{equation*}
  \frac{u''(r)}{u'(r)} \frac{\psi(r)}{\psi'(r)}>M \qquad \text{for any } r>r_M \, .
\end{equation*}
Multiplying both sides of the above inequality by $\psi'(r)/\psi(r)$ and integrating we obtain
\begin{equation*}
  \log\left(\frac{|u'(r)|}{|u'(r_M)|}\right)>M \log\left(\frac{\psi(r)}{\psi(r_M)}\right) \qquad \text{for any } r>r_M
\end{equation*}
from which it follows that
\begin{equation*}
   |u'(r)|>\frac{|u'(r_M)|}{(\psi(r_M))^M} \, (\psi(r))^M \qquad \text{for any } r>r_M \, .
\end{equation*}
Passing to the limit as $r\to +\infty$ and recalling \eqref{A5}, we conclude that $|u'(r)|\to +\infty$ as $r \to +\infty$ in contradiction with Lemma \ref{l:u'-0}. This completes the proof of $(ii)$.
\end{proof}

\noindent
{\bf End of the proof of Theorem \ref{final}.} % \label{asym_conclusion}
The proof of $(i)$ in the case $\frac{\psi}{\psi'}\in L^1(0,\infty)$ is an immediate consequence of Lemma \ref{l:integrable} -$(i)$.

Let us proceed with the proof of $(ii)$. First of all by Lemma \ref{l:integrable}-$(ii)$ we have that $u$ diverges to $-\infty$ as $r\to +\infty$. In this way \eqref{eq:lim-u-infty} is satisfied, hence Lemma \ref{sol_asymptotic_model_infty} completes the proof of Theorem \ref{final}.
\endproof

%%%%%%%%%%%%%%%%%%%%%%%%%%%%%%%%%%%%%%%%%%%%%%%%%%%%%%%%%%

\section{Preliminary results about stability and intersection properties}\label{5}

% \section{Proof of Theorem \ref{stability statement}}

%
We start by stating an equivalent characterization of stability in the case of radial solutions of \eqref{main_eq}. Since the proof can be obtained by following the proof of \cite[Lemma 5.1]{elvise-14} with obvious changes we omit it.

\begin{lemma}
	Let $\psi$ satisfy \eqref{A1}-\eqref{A3} and let $u$ be a radial solution of \eqref{main_eq}. Then $u$ is stable if and only if
	\begin{align}\label{radial_con}
		\int_{0}^{+\infty}(\chi^\prime(r))^2 \, (\psi(r))^{N-1}\dr
         -\int_{0}^{+\infty}e^{u(r)}(\chi(r))^2 \, (\psi(r))^{N-1}\dr\geq 0,
	\end{align}
for every radial function $\chi\in C_c^\infty(M)$.
\end{lemma}

We now give the statement of a series of lemmas that will be exploited in the proofs of the main results about stability and intersection properties.

In the sequel, we denote by $u_\alpha$ the unique solution of \eqref{regular_ode_model} with $\alpha=u_\alpha(0)$ and we consider the set
\begin{equation} \label{eq:def-S}
   \mathcal S:=\{\alpha\in \R: u_\alpha \text{ is stable}\} \, .
\end{equation}

%\begin{lemma}\label{less_initial}
%	Let $\psi$ satisfy \eqref{A1}-\eqref{A2}, and let $u_\alpha$ be the solution to \eqref{regular_ode_model} for $\alpha \in \R$. Then $u_\alpha(r)\leq \alpha$ for any $r\in [0,+\infty)$.
%\end{lemma}
%\begin{proof}
%	It follows from Proposition \ref{first_facts}.
%	\end{proof}
	
Before proceeding, we recall the variational characterization of the bottom of the $L^2$ spectrum of $-\Delta_g$ in $M$:

\begin{equation}\label{lambda1}
\lambda_1(M):= \inf_{\varphi \in C^{\infty}_{c}(M)\setminus
\{0\}}\frac{\int_{M} |\nabla_g \varphi |_{g}^2\,dV_{g}}{\int_{M}
\varphi^2\,dV_{g}} \, .
\end{equation}

It is well known that under assumptions \eqref{A1}-\eqref{A3} we have that $\lambda_1(M)>0$, see e.g.,
\cite[Lemma 4.1]{elvise-14}. The bottom of the spectrum is involved in the stability of solutions $u_\alpha$ of \eqref{regular_ode_model} for sufficiently small values of $\alpha$. Indeed, we prove:

\begin{lemma}\label{l6.3}
Let $\psi$ satisfy  \eqref{A1}-\eqref{A3} and let $u_\alpha$ be a solution of \eqref{regular_ode_model} with $\alpha\leq \log(\lambda_1(M))$. Then $u_\alpha$ is stable and in particular, the set $\mathcal S$ is not empty.
\end{lemma}

\begin{proof}
Using \eqref{lambda1} and Proposition \ref{first_facts} we infer
\begin{align*}
	  \int_M |\gradg v|_{g}^2\, \dvg
	\geq \frac{\lambda_1(M)}{e^\alpha} \int_M e^\alpha  v^2\,\dvg\geq \frac{\lambda_1(M)}{e^\alpha} \int_M e^{u_\alpha(r)} v^2\,\dvg \qquad \text{for any } v\in C^\infty_c(M)
\end{align*}
which gives the stability of $u_\alpha$ if $\alpha\leq \log(\lambda_1(M))$.
\end{proof}

The next step is to prove that the set $\mathcal S$ is an interval.
%Now, for any $\alpha>\beta$, we define
%\begin{align*}
%	\zeta_{\alpha\beta}:=\sup\{r\in (0,\infty)\,\,|\,\, u_\alpha(s)>u_\beta(s) \text{ for any }s\in (0,r)\}.
%\end{align*}
%If $\zeta_{\alpha\beta}<+\infty$ then $u_\alpha-u_\beta$ first vanishes in $\zeta_{\alpha\beta}$. Next we prove:
First, we state three preliminary lemmas; the fact that $\mathcal S$ is an interval will be proved in the fourth one (Lemma \ref{l6.7} below).

\begin{lemma}\label{l4.4}
Let $\psi$ satisfy \eqref{A1}-\eqref{A3} and let $a,\,b,\,R\in \mathbb{R}$ be such that $b>a$. If $u_\alpha(r)$ is the solution of \eqref{regular_ode_model} with $\alpha\in [a,b]$, then for any $r\in[0,R]$, the map $\alpha\mapsto u(\alpha,r):=u_\alpha(r)$ is differentiable in $ [a,b]$ and for any $\alpha_0\in [a,b]$
\begin{align}\label{l6.4_id}
	\lim_{\alpha\rightarrow\alpha_0}\,\sup_{r\in [0,R]} \bigg|\frac{\partial u}{\partial \alpha}(\alpha,r)-\frac{\partial u}{\partial \alpha}(\alpha_0,r)\bigg|=0.
\end{align}
Moreover for any $\alpha\in[a,b]$ and $r\in[0,R]$, the function $v_{\alpha}(r):=\frac{\partial u}{\partial \alpha}(\alpha,r)$ is the solution of the following problem
\begin{equation}\label{l6.4_use1}
	\begin{dcases}
		-v^{\prime\prime}(r)-(N-1)\frac{\psi^\prime(r)}{\psi(r)}v^\prime(r)=e^{u_{\alpha}(r)}v(r) &  \\
		v(0)=1  & \\
		v^\prime(0)=0. & \\
	\end{dcases}
\end{equation}
\end{lemma}

\begin{proof} The proof can be obtained by proceeding along the lines of \cite[Lemma 5.6]{elvise-14} with suitable changes concerning essentially the fact that the power nonlinearity is replaced here by the exponential nonlinearity. For the sake of completeness, we recall the main steps here below.

For any $r\in [0,R]$ and $\alpha\in [a,b]$, let us define
	\begin{align*}
		w(r)=\frac{u_\alpha(r)-u_{\alpha_0}(r)}{\alpha-\alpha_0}-v_{\alpha_0}(r) \text{ and } z(r)=w^\prime(r)
	\end{align*}
where $v_{\alpha_0}$ is the solution of problem \eqref{l6.4_use1} with $\alpha=\alpha_0$.

Then
\begin{align}  \label{eq:z} z^{\prime}(r)+(N-1)\frac{\psi^\prime(r)}{\psi(r)}z(r)=-\bigg(\frac{e^{u_\alpha(r)}-e^{u_{\alpha_0}(r)}}{\alpha-\alpha_0}
  -e^{u_{\alpha_0}(r)}v_{\alpha_0}(r)\bigg) \, .
\end{align}
For any $\delta>0$, $\alpha\in (\alpha_0-\delta,\alpha_0+\delta)\cap [a,b]$ and $r\in [0,R]$, we have
\begin{align*}
	\bigg|\frac{e^{u_\alpha(r)}-e^{u_{\alpha_0}(r)}}{\alpha-\alpha_0}-e^{u_{\alpha_0}(r)}v_{\alpha_0}(r)\bigg|&\leq \bigg|\frac{e^{u_\alpha(r)}-e^{u_{\alpha_0}(r)}}{\alpha-\alpha_0}-\frac{u_\alpha(r)-u_{\alpha_0}(r)}{\alpha-\alpha_0} \, e^{u_{\alpha_0}(r)}+e^{u_{\alpha_0}(r)}w(r)\bigg|\\&\leq \frac{|u_\alpha(r)-u_{\alpha_0}(r)|}{|\alpha-\alpha_0|}|e^{\xi(r)}-e^{u_{\alpha_0}(r)}|+e^{u_{\alpha_0}(r)}|w(r)|
\end{align*}
where, by Lagrange Theorem, $\min\{u_\alpha(r),u_{\alpha_0}(r)\}<\xi(r)<\max \{u_\alpha(r),u_{\alpha_0}(r)\}$ any $r\in [0,R]$.
Recalling that for any $\alpha$ the functions $u_\alpha$ are decreasing and using again Lagrange Theorem, we obtain

\begin{align*}
	& \bigg|\frac{e^{u_\alpha(r)}-e^{u_{\alpha_0}(r)}}{\alpha-\alpha_0}-e^{u_{\alpha_0}(r)}v_{\alpha_0}(r)\bigg|\leq e^{\alpha_0+\delta} \, \frac{|u_\alpha(r)-u_{\alpha_0}(r)|^2}{|\alpha-\alpha_0|}+e^{\alpha_0}|w(r)| \\[10pt]
& \qquad \leq |w(r)| \left\{ e^{\alpha_0+\delta}\, |u_\alpha(r)-u_{\alpha_0}(r)|+e^{\alpha_0}\right\}+e^{\alpha_0+\delta} \, |v_{\alpha_0}(r)| \, |u_\alpha(r)-u_{\alpha_0}(r)|.
\end{align*}
Now, by continuous dependence, for any $\varepsilon>0$ we may choose $\delta$ small enough in such a way that $\sup_{r\in [0,R]}|u_\alpha(r)-u_{\alpha_0}(r)|<\varepsilon$ and we obtain for $r\in [0,R]$,
\begin{align*}
	\bigg|\frac{e^{u_\alpha(r)}-e^{u_{\alpha_0}(r)}}{\alpha-\alpha_0}-e^{u_{\alpha_0}(r)}v_{\alpha_0}(r)\bigg|
\leq (1+\varepsilon) e^{\alpha_0+\delta} |w(r)|+C\varepsilon
\end{align*}
where we put $C=e^{\alpha_0+\delta} \sup_{r\in [0,R]} |v_{\alpha_0}(r)|$.
Furthermore, observing that $w(0)=0$, we infer
\begin{align} \label{eq:increment}
	\bigg|\frac{e^{u_\alpha(r)}-e^{u_{\alpha_0}(r)}}{\alpha-\alpha_0}-e^{u_{\alpha_0}(r)}v_{\alpha_0}(r)\bigg|\leq (1+\varepsilon) e^{\alpha_0+\delta} \int_0^r|z(s)|{\rm d}s+C\varepsilon.
\end{align}
Combining \eqref{eq:z} and \eqref{eq:increment} and observing that $z(0)=0$, we obtain
\begin{align*}
	|z(r)|\leq K (1+\varepsilon) e^{\alpha_0+\delta} \int_0^r|z(s)|{\rm d}s+KC\varepsilon,
\end{align*}
for any $r\in[0,R]$ and $\alpha\in (\alpha_0-\delta,\alpha_0+\delta)\cap [a,b]$, where we put $K:=\sup_{r\in [0,R]}\frac{\int_{0}^r (\psi(s))^{N-1}{\rm d}s}{(\psi(r))^{N-1}}$.

Finally, exploiting standard Gronwall-type estimates we obtain
\begin{align*}
	\lim_{\alpha\rightarrow\alpha_0}\sup_{r\in [0,R]}|z(r)|=0 \space \implies \lim_{\alpha\rightarrow\alpha_0}\sup_{r\in [0,R]}|w(r)|=0.
\end{align*}
This proves the differentiability with respect to $\alpha$ of the map $\alpha\mapsto u(\alpha,r)$ and shows that the derivative with respect to $\alpha$ is a solution of \eqref{l6.4_use1}. The proof of \eqref{l6.4_id} is a consequence of a standard continuous dependence result for the Cauchy problem \eqref{l6.4_use1}.
\end{proof}

\begin{lemma} \label{l:5.7}
	Let	$\psi$ satisfy \eqref{A1}-\eqref{A3} and let $\alpha_1>\alpha_2\geq \alpha_3>\alpha_4$. Then the first intersection between $u_{\alpha_1}$ and $u_{\alpha_2}$ cannot take place after the first intersection between $u_{\alpha_3}$ and $u_{\alpha_4}$.
\end{lemma}

\begin{proof} We follow closely the proof of \cite[Lemma 7.3]{bonforte-13}. We divide the proof into two steps.

{\bf Step 1.} We first prove the lemma when only three functions $u_{\alpha_1}, u_{\alpha_2}, u_{\alpha_3}$ with $\alpha_1>\alpha_2>\alpha_3$, are involved. In other words, we prove that the first intersection between $u_{\alpha_1}$ and $u_{\alpha_2}$ cannot take place after the first intersection between $u_{\alpha_2}$ and $u_{\alpha_3}$.

Let $w_1=u_{\alpha_1}-u_{\alpha_2}$ and $w_2=u_{\alpha_2}-u_{\alpha_3}$. Then $w_i(0)>0$ and $w_i^\prime(0)=0$ for $i=1,2$. Let $r_1>0$ be such that $w_i$ has no zero in $[0,r_1]$. Then for $i=1,2$, the functions $w_i$ satisfy
\begin{align*}
	w_i^{\prime\prime}(r)+(N-1)\frac{\psi^\prime(r)}{\psi(r)} \, w_i^\prime(r)+b_i(r)w_i(r)=0,
\end{align*}
where the functions $b_i$ are positive in $[0,r_1]$ and they satisfy
\begin{align*}
   & u_{\alpha_2}(r)<\log(b_1(r))<u_{\alpha_1}(r) \quad \text{and} \quad  u_{\alpha_3}(r)<\log(b_2(r))<u_{\alpha_2}(r)
   \qquad \text{for any } r\in [0,r_1] \, ,
\end{align*}
thanks to Lagrange Theorem.
In particular this gives $b_1(r)> b_2(r)$ in $[0,r_1]$. Putting $z=w_1/w_2$ we have that
\begin{align*}
   & w_2(r) z''(r)+\left[2w_2'(r)+(N-1)\frac{\psi'(r)}{\psi(r)} \, w_2(r)\right] z'(r)=-z(r) (b_1(r)-b_2(r))w_2(r) <0
\end{align*}
for any $r\in [0,r_1]$ and moreover $z(0)>0$ and $z'(0)=0$.
If we set $a(r):=2\frac{w_2'(r)}{w_2(r)}+(N-1)\frac{\psi'(r)}{\psi(r)}$, the above inequality can be written as $z''(r)+a(r)z'(r)<0$ in $[0,r_1]$, being $w_2>0$ in $[0,r_1]$. Then, for $\varepsilon >0$, multiplying both sides by $e^{\int_{\varepsilon}^r a(t)\,dt}$ and integrating in $[\varepsilon,r]$, we get
$$
    z'(r)\leq z'(\varepsilon)\, e^{-\int_{\varepsilon}^r a(t)\,dt} \quad \text{for all }r\in  (\varepsilon,r_1] \, .
$$
By \eqref{A1} and the fact that $w_2(0)>0$ and $w_2'(0)=0$ we deduce that $a(r)\sim \frac{N-1}{r}$ as $r \rightarrow 0^+$. Then, letting $\varepsilon \rightarrow 0^+$ in the above inequality and recalling that $z'(\varepsilon)\to z'(0)= 0$, we conclude that $z'\leq 0$ in $(0,r_1]$ so that $z$ is non increasing in the same interval and in particular $z(r)\leq z(0)$ for any $r\in (0,r_1]$. This completes the proof of Step 1. Indeed, if $\zeta_i$ is the first zero of $w_i$, $i=1,2$ and if we assume by contradiction that $\zeta_1>\zeta_2$ then we may apply the previous estimate for any $0< r_1<\zeta_2$ and obtain
\begin{equation*}
   w_1(r)\le z(0) w_2(r) \qquad \text{for any } r\in (0,\zeta_2) \, .
\end{equation*}
Then, letting $r\to \zeta_2^-$ we conclude that $w_1(\zeta_2)\le 0$ in contradiction with $\zeta_1>\zeta_2$.

{\bf Step 2.} We complete here the proof of the lemma. We denote by $\zeta_{ij}$ the first intersection between $u_{\alpha_i}$ and $u_{\alpha_j}$ with $i\le j$. We have to prove that $\zeta_{12}\le \zeta_{34}$. We apply Step 1 twice, first to prove that $\zeta_{12}\le \zeta_{23}$ and then to prove that $\zeta_{23}\le \zeta_{34}$. The combination of the two inequalities readily completes the proof of Step 2.
\end{proof}

\begin{lemma}{\cite[Lemma 5.8]{elvise-14}} \label{l:5.8}
Let	$\psi$ satisfy \eqref{A1}-\eqref{A3}. Then $\lim_{r\rightarrow0^+}\, \lambda_1(B_r)=+\infty$.
\end{lemma}

By combining the above lemmas one gets:
\begin{lemma}\label{l6.7}
Let	$\psi$ satisfy \eqref{A1}-\eqref{A3} and assume $\alpha>\beta$. If $u_\beta$ is unstable then $u_\alpha$ is also the unstable solution. In particular, the set $\mathcal S$ defined in \eqref{eq:def-S} is an interval.
\end{lemma}

\begin{proof}
Suppose that $u_\alpha$ and $u_\beta$ have no intersections. Then $u_\alpha(r)>u_\beta(r)$ for any $r\geq 0$ and $e^{u_\alpha(r)}>e^{u_\beta(r)}$ for any $r\geq 0$. Suppose by contradiction that $u_\alpha$ is stable, then \eqref{radial_con} implies
\begin{align*}
	\int_{0}^{+\infty}(\chi^\prime(r))^2 \, (\psi(r))^{N-1}\dr\geq \int_{0}^{+\infty}e^{u_\alpha(r)}(\chi(r))^2 \, (\psi(r))^{N-1}\dr
	\\ \geq \int_{0}^{+\infty}e^{u_\beta(r)}(\chi(r))^2 \, (\psi(r))^{N-1}\dr,
\end{align*}
for every radial function $\chi\in C_c^\infty(M)$. This implies the stability of $u_\beta$, a contradiction.

We may suppose now that $u_\alpha$ and $u_\beta$ have at least one intersection. Exploiting Lemma \ref{l4.4}, Lemma \ref{l:5.7} and Lemma \ref{l:5.8}, we may follow the part of the proof of \cite[Lemma 5.9]{elvise-14} in which $u_\alpha$ and $u_\beta$ have at least one intersection, and prove the instability of $u_\alpha$.
\end{proof}

\subsection{Low dimensions}

In the next lemma we prove that if $2\le N\le 9$, the set $\mathcal S$ is bounded above. To this aim we first set
\begin{align} \label{eq:def-eta}
	\eta:=\sup \mathcal S\,.
\end{align}

\begin{lemma}\label{t6.2}
	Let $\psi$ satisfy \eqref{A1}-\eqref{A3} and let $2\leq N \leq 9$. Then there exists $\alpha_0>0$ such that for any $\alpha>\alpha_0$, the solution $u_\alpha$ of \eqref{regular_ode_model} is unstable. In particular, the number $\eta$ defined in \eqref{eq:def-eta} is finite.
\end{lemma}

\begin{proof}
	We adapt to our framework the blow-up arguments of \cite[Lemmas 4.9 and 5.5]{elvise-14} and
	 \cite[Lemma 7.1]{bonforte-13}. Since we know from Lemma \ref{l6.7} that the set $\mathcal S$ is an interval, we proceed by contradiction assuming that  $u_\alpha$ is stable for any $\alpha \in \R$.

Let us define $u_\lambda$ as the solution of \eqref{regular_ode_model} with initial condition $\alpha=\log(\frac{e}{\lambda^2})$ and define
	\begin{align*}
		v_\lambda(s)=u_\lambda(\lambda s)+2\log(\lambda).
	\end{align*}
Now, one can check that $v_\lambda(0)=1$ and it satisfies
\begin{align} \label{eq:v-lambda}
	v_\lambda^{\prime\prime}(s)+\frac{N-1}{s}\frac{\psi^\prime(\lambda s)}{\psi(\lambda s)}\, \lambda s \, v'_\lambda(s)+e^{v_\lambda(s)}=0.
\end{align}
Furthermore, using the assumptions on $\psi$, we can find that for any fixed $S>0$,
\begin{align} \label{eq:l-s-psi}
	 \frac{\psi^\prime(\lambda s)}{\psi(\lambda s)}\,\lambda s \to 1 \qquad
      \text{as $\lambda \to 0^+$, uniformly in $(0,S]$. }
\end{align}
If we define $F_\lambda(r)=\frac 12 (u_\lambda'(r))^2+e^{u_\lambda(r)}$ then we know from the proof of Proposition \ref{first_facts} that $F_\lambda$ is decreasing and hence
\begin{align*}
   & |u_\lambda'(r)|^2= 2\Big[F_\lambda(r)-e^{u_\lambda(r)}\Big]\le 2\Big[e^{u_\lambda(0)}-e^{u_\lambda(r)}\Big]
     \le 2e^{u_\lambda(0)} \Big[u_\lambda(0)-u_\lambda(r)\Big]
\end{align*}
from which we obtain
\begin{align*}
   & |u_\lambda'(r)|\le \frac{\sqrt{2e}}{\lambda} \, \left(\int_{0}^{r} |u_\lambda'(t)| \, {\rm d}t \right)^{1/2}
   \qquad \text{for any } r>0
\end{align*}
where we recall that $e^{u_\lambda(0)}=\frac{e}{\lambda^2}$.

By Gronwall-type estimates we obtain
\begin{equation*}
  |u_\lambda'(r)|\le \frac{e}{\lambda^2} \, r \qquad \text{for any } r>0
\end{equation*}
and from the definition of $v_\lambda$ we finally obtain
\begin{equation} \label{eq:v-lambda-bis}
  |v_\lambda'(s)|\le e s \qquad \text{for any } s>0 \, .
\end{equation}
Combining \eqref{eq:v-lambda}, \eqref{eq:l-s-psi}, \eqref{eq:v-lambda-bis}
we deduce that for any fixed $S>0$ there exists $\bar \lambda(S)>0$ such that
$$
  |v_\lambda^{\prime\prime}(s)|\leq 2(N-1) e+e^{v_\lambda(0)}=(2N-1)e \quad \text{for any $s\in [0,S]$ and $0<\lambda<\bar \lambda(S)$} \, .
$$
Hence, using Ascoli-Arzel\'a Theorem on $[0,S]$, we deduce that there exists $\bar{v}\in C^1([0,S])$ such that $v_\lambda \rightarrow \bar{v}$ in $C^1([0,S])$ as $\lambda\rightarrow 0^+$ and $\bar{v}$ satisfies
\begin{align*}
	\bar{v}^{\prime\prime}(s)+\frac{N-1}{s} \, \bar{v}'(s)+e^{\bar{v}(s)}=0 \space \text{ and } \bar{v}(0)=1.
\end{align*}

Since $u_\lambda$ is a stable radial solution, using \eqref{radial_con} we have that
\begin{align*}
   \int_{0}^{+\infty}(\chi^\prime(r))^2 \, (\psi(r))^{N-1}\dr-\int_{0}^{+\infty}e^{u_\lambda(r)}(\chi(r))^2 \, (\psi(r))^{N-1}\dr\geq 0 \, ,
\end{align*}
for every radial function $\chi\in C_c^\infty(M)$. In terms of $v_\lambda$ the above inequality reads
\begin{align*}
	\int_{0}^{+\infty}(\chi^\prime(r))^2(\psi(r))^{N-1}\dr
    -\frac{1}{\lambda^2}\int_{0}^{+\infty}e^{v_\lambda(\frac{r}{\lambda})}(\chi(r))^2 \, (\psi(r))^{N-1}\dr\geq 0\,,
\end{align*}
for every radial function $\chi\in C_c^\infty(M)$. Now, choosing $\eta_\lambda(r):=\eta(\frac{r}{\lambda})$ as test function for some radial function $\eta\in C_c^\infty(M)$ and then using the change of variable $s=\frac{r}{\lambda}$ we infer
\begin{align*}
  \int_{0}^{+\infty}(\eta^\prime(s))^2(\psi(s\lambda))^{N-1}\ds
     -\int_{0}^{+\infty}e^{v_\lambda(s)}(\eta(s))^2 \, (\psi(s\lambda))^{N-1}\ds\geq 0,
\end{align*}
for every radial function $\eta\in C_c^\infty(M)$. Now fix $S> 0$ and then choose $\eta$ in such a way that supp$(\eta)\subset B_S$ where $B_S$ denotes the geodesic ball of radius $S$ centered at the pole $o$. By Lagrange Theorem for any $s\in [0,S],$ there exist $0<\zeta<s\lambda$ and $0<|t|<\frac{|\psi^{\prime\prime}(\zeta)|}{2}(s\lambda)$ such that for $\lambda\rightarrow 0^+$
\begin{align*}
	(\psi(s\lambda))^{N-1}=(s\lambda)^{N-1}+g(\zeta,t)(s\lambda)^{N},
\end{align*}
where $g(\zeta,t)=(N-1)(1+t)^{N-2} \, \frac{\psi^{\prime\prime}(\zeta)}{2}$. This gives, after cancelling  $\lambda^{N-1}$, that
\begin{align*}
&\int_{0}^{+\infty}(\eta^\prime(s))^2s^{N-1}\ds+\int_{0}^{+\infty}(\eta^\prime(s))^2 \, s^{N}g(\zeta,t)\lambda\ds\\&-\int_{0}^{+\infty}e^{v_\lambda(s)}(\eta(s))^2s^{N-1}\ds
-\int_{0}^{+\infty}e^{v_\lambda(s)}(\eta(s))^2 \, s^{N}g(\zeta,t)\lambda\ds\geq 0.
\end{align*}
Therefore taking the limit as $\lambda\rightarrow 0^+$, we finally obtain
\begin{align*}
\int_{0}^{+\infty}(\eta^\prime(s))^2s^{N-1}\ds	-\int_{0}^{+\infty}e^{\bar{v}(s)}(\eta(s))^2s^{N-1}\ds\geq 0,
\end{align*}
for every radial function $\eta\in C_c^\infty(M)$. Therefore $\bar{v}$ is a stable solution of the equation $-\Delta u=e^u$ on $\mathbb{R}^N$ for $2\leq N\leq 9$. This contradicts the results of \cite{farina2} and concludes the proof.
\end{proof}

We complete the stability picture for $2\leq N \leq 9$ by showing that $\mathcal{S}$ is a closed interval.
\begin{lemma}\label{l6.8}
	Let $\psi$ satisfy \eqref{A1}-\eqref{A3} and let $2\leq N \leq 9$, then $\mathcal{S}=(-\infty,\eta]$.
\end{lemma}
\begin{proof}
	Suppose $\eta\notin \mathcal{S}$ and so $u_\eta$ is unstable solution. Then for each $n\in \mathbb{N}$ there exist $\eta_n\in \mathcal{S}$ such that $\eta_n\rightarrow\eta$ for $n\rightarrow +\infty$. The definition of $\eta_n$ gives that each $u_{\eta_n}$ is a stable solution. Hence for each radial function $\chi\in C_c^\infty(M)$ there holds
		\begin{align*}
		\int_{0}^{+\infty}(\chi^\prime(r))^2(\psi(r))^{N-1}\dr-\int_{0}^{+\infty}e^{u_{\eta_n}(r)}(\chi(r))^2(\psi(r))^{N-1}\dr\geq 0.
	\end{align*}
Now, for any function $\chi\in C_c^\infty(M)$, we have that supp($\chi$) is a compact set and continuous dependence on initial data gives $u_{\eta_n}\rightarrow u_\eta$ uniformly. This implies
	\begin{align*} \int_{0}^{+\infty}(\chi^\prime(r))^2(\psi(r))^{N-1}\dr-\int_{0}^{+\infty}e^{u_{\eta}(r)}(\chi(r))^2(\psi(r))^{N-1}\dr\geq 0.
\end{align*}
contradicting the fact that $u_\eta$ is an unstable solution.
\end{proof}

By Lemma \ref{l6.3} we know that $\eta \geq \log(\lambda_1(M))$; the next two lemmas allow us to improve these bounds when $\psi/\psi'$ is not integrable.
\begin{lemma}\label{big_lambda}
Let $\psi$ satisfy \eqref{A1}-\eqref{A3}. Then for any $\alpha$,
\begin{align*}
	\Lambda(M,\alpha):=\,\, \inf_{v\in H^1(M)\setminus\{0\}}\frac{\int_M |\gradg v|_g^2\, \dvg}{\int_M e^{u_\alpha}v^2\,\dvg}
\end{align*}
admits a minimizer.
\end{lemma}
\begin{proof}
	The proof is a straightforward adaptation of \cite[Lemma 5.11]{elvise-14} to this setting, therefore we omit it.
\end{proof}

\begin{lemma} \label{l:eta>}
Let $\psi$ satisfy \eqref{A1}-\eqref{A3} and assume that $\frac{\psi}{\psi^\prime}\notin L^1(0,+\infty)$. Then $\eta>\log(\lambda_1(M))$.
\end{lemma}
\begin{proof}
	 Set $\bar{\alpha}:=\log(\lambda_1(M))$ and let $\Lambda(M,\alpha)$ be as in the statement of Lemma \ref{big_lambda}, then $\Lambda(M,\bar{\alpha})>1$. Indeed, if $w\in H^1(M)$ is the minimizer of $\Lambda(M,\bar{\alpha})$, then
	\begin{align*}
		\Lambda(M,\bar{\alpha})=\frac{\int_M |\gradg w|_g^2\, \dvg}{\int_M e^{u_{\bar{\alpha}}}|w|^2\,\dvg}\geq \frac{\int_M |\gradg w|_g ^2\, \dvg}{e^{\bar{\alpha}}\int_M |w|^2\,\dvg}\geq 1.
		\end{align*}
Therefore, if $\Lambda(M,\bar{\alpha})=1$, $w$ solves
\begin{align*}
	-\Delta_{g}w=\lambda_1(M)w\,\, \text{ in }M \,\,\text{ and } -\Delta_{g}w=e^{u_{\bar{\alpha}}}w \,\, \text{ in }M.
\end{align*}
This implies that $u_{\bar{\alpha}}$ is constant, which is a contradiction. Hence, $\Lambda(M,\bar{\alpha})>1$.

Next we consider a decreasing sequence $\{\alpha_{k}\}$ such that $\alpha_{k} \rightarrow \bar{\alpha}$ as $k \rightarrow +\infty$. We will show that $\Lambda(M,\alpha_{k})>1$ definitively, hence $\eta>\log(\lambda_1(M))$ which is the thesis.

By contradiction, assume that there exists $K>0$ such that $\Lambda(M,\alpha_{k})\leq 1$ for $k>K$. For any $k$ let $w_k$ be a minimizer of $\Lambda(M,\alpha_{k})$ satisfying $\int_{M}e^{u_{\alpha_{k}}}w_k^2\, \dvg=1$. The sequence $\{w_k\}$ is bounded in $H^1(M)$ hence, up to a subsequence, $w_k \rightharpoonup w$ in $H^1(M)$ as $k \rightarrow +\infty$ and $w_k \to w$ strongly in $L^2(B_R)$ for any $R>0$, where we recall that $B_R$ denotes the geodesic ball of radius $R$ centered at the pole $o$.

For any $\alpha \in \R$ consider the Lyapunov functional \eqref{lyapunov_functional_model}: $F_\alpha(r)=\frac{1}{2}(u_\alpha^\prime(r))^2+e^{u_\alpha(r)}$ for $r>0$. From the proof of Theorem \ref{final}, the stated assumptions imply that $\lim_{r \rightarrow + \infty} u_{{\alpha}}(r)=-\infty$ and $\lim_{r \rightarrow + \infty} u'_{{\alpha}}(r)=0$. Therefore, for every $\varepsilon>0$ there exists $R_\varepsilon>0$ such that $F_{\bar{\alpha}}(R_\varepsilon)<\varepsilon$. Since for any $r>0$, $u_{\alpha_k}(r)\rightarrow u_{\bar{\alpha}}(r)$ and $u^\prime_{\alpha_k}(r)\rightarrow u^\prime_{\bar{\alpha}}(r)$ as $k \rightarrow +\infty$, then there exists $\bar K= \bar K(\varepsilon)>0$ such that $F_{\alpha_{k}}(R_\varepsilon)<\varepsilon$ for $k>\bar K$. Since the functions $F_{\alpha_k}$ are nonincreasing, it follows that $F_{\alpha_{k}}(r)<\varepsilon$ and $e^{u_{\alpha_k}(r)}<\varepsilon$ for any $r\geq R_\varepsilon$ and $k>\bar K$. Therefore,
\begin{align*}
	&\bigg|\int_M e^{u_{\alpha_{k}}}w_k^2 \, \dvg-\int_M e^{u_{\bar{\alpha}}}w^2 \, \dvg\bigg|\\&\leq 	\bigg|\int_{B_{R_\varepsilon}} e^{u_{\alpha_{k}}}w_k^2 \, \dvg-\int_{B_{R_\varepsilon}} e^{u_{\bar{\alpha}}}w_k^2 \, \dvg\bigg|+\bigg|\int_{B_{R_\varepsilon}} e^{u_{\bar{\alpha}}}w_k^2 \, \dvg-\int_{B_{R_\varepsilon}} e^{u_{\bar{\alpha}}}w^2 \, \dvg\bigg|\\&+\bigg|\int_{M\setminus {B_{R_\varepsilon}}} e^{u_{\alpha_{k}}}w_k^2 \, \dvg-\int_{M\setminus {B_{R_\varepsilon}}} e^{u_{\bar{\alpha}}}w^2 \, \dvg\bigg|\\&\leq \sup_{{B_{R_\varepsilon}}}|e^{u_{\alpha_k}}-e^{u_{\bar{\alpha}}}|\int_{B_{R_\varepsilon}}w_k^2 \, \dvg+\bigg|\int_{B_{R_\varepsilon}} e^{u_{\bar{\alpha}}}w_k^2 \, \dvg-\int_{B_{R_\varepsilon}} e^{u_{\bar{\alpha}}}w^2 \, \dvg\bigg|\\&+\varepsilon \int_{M\setminus {B_{R_\varepsilon}}} w_k^2 \, \dvg+\varepsilon \int_{M\setminus {B_{R_\varepsilon}}} w^2 \, \dvg\\&\leq \sup_{{B_{R_\varepsilon}}}|e^{u_{\alpha_k}}-e^{u_{\bar{\alpha}}}|\int_{B_{R_\varepsilon}} \!\! w_k^2 \, \dvg+o(1)+\frac{\varepsilon}{\lambda_1(M)}\bigg(\int_M|\nabla_g w_k|_g ^2\dvg+\liminf_{k \rightarrow + \infty}\int_M|\nabla_g w_k|_g ^2\dvg\bigg)\\&
	\leq C \varepsilon+\frac{\varepsilon}{\lambda_1(M)}\bigg(\Lambda(M,\alpha_{k})+\liminf_{k \rightarrow + \infty}\Lambda(M,\alpha_{k})\bigg)+o(1) \leq C \varepsilon+\frac{2 \varepsilon}{\lambda_1(M)}+o(1) \, .
\end{align*}
In the above estimate we used the following facts: $w_k\rightarrow w$ in $L^2(B_{R_\varepsilon})$ strongly, $e^{u_{\alpha_{k}}}\rightarrow e^{u_{\bar{\alpha}}}$ uniformly, the lower semicontinuity of the $H^1(M)$-norm and the inequality $\Lambda(M,\alpha_{k})\leq 1$. Letting $k\to +\infty$, since $\varepsilon$ was chosen arbitrarily, we conclude that
\begin{align*}
\lim_{k \rightarrow + \infty} \int_M e^{u_{\alpha_{k}}}w_k^2 \, \dvg=\int_M e^{u_{\bar{\alpha}}}w^2\dvg.
\end{align*}
Using again the lower semicontinuity of the $H^1(M)$-norm, we finally have
\begin{align*}
	1<\Lambda(M,\bar{\alpha})\leq \frac{\int_M |\gradg w|_g^2\, \dvg}{\int_M e^{u_{\bar{\alpha}}}|w|^2\,\dvg}\leq\liminf_{k \rightarrow + \infty} \frac{\int_M |\gradg w_k|_g^2\, \dvg}{\int_M e^{u_{\alpha_k}}|w_k|^2\,\dvg}=\liminf_{k \to +\infty} \Lambda(M,\alpha_{k}),
\end{align*}
a contradiction. This ensures that $\Lambda(M,\alpha_{k})>1$ definitively and concludes the proof.
\end{proof}

\subsection{High dimensions}

   We are going to state the key ingredients in the proofs of Theorem \ref{t:stability} and Theorem \ref{intersection}-$(iii)$. Inspired by  \cite{hasegawa} and \cite{joseph-73} (see the Appendix), given a radial regular solution $u$ of (2.2), we define the function
   \begin{equation} \label{eq:def-v}
     v(r)=u(r)+2\log(\psi(r))-\log[2(N-2] \, .
   \end{equation}
   Then $v$ solves the equation
   \begin{equation} \label{eq:v-psi}
     v''(r)+(N-1)\frac{\psi'(r)}{\psi(r)} \, v'(r)+2(N-2)\left[\frac{e^{v(r)}}{(\psi(r))^2}-\left(\frac{\psi'(r)}{\psi(r)}\right)^2\right]-\frac{2\psi''(r)}{\psi(r)}=0
   \end{equation}

   The linearized equation at a function $v=v(r)$ becomes
    \begin{equation*}
     \varphi''(r)+(N-1)\frac{\psi'(r)}{\psi(r)} \, \varphi'(r)+\frac{2(N-2)}{(\psi(r))^2} \, e^{v(r)} \, \varphi(r)=0 \, .
   \end{equation*}

   In particular, linearizing at $v(r)=2\log(\psi'(r))$, we obtain the equation
   \begin{equation} \label{eq:linearized}
     \varphi''(r)+(N-1)\frac{\psi'(r)}{\psi(r)} \, \varphi'(r)+2(N-2) \left(\frac{\psi'(r)}{\psi(r)}\right)^2 \, \varphi(r)=0 \, .
   \end{equation}
   Next we define the operator $\mathcal L$ as the left hand sight of \eqref{eq:linearized}, i.e.
   \begin{equation}\label{eq:def-L}
      \mathcal L\varphi(r)=\varphi''(r)+(N-1)\frac{\psi'(r)}{\psi(r)} \, \varphi'(r)+2(N-2) \left(\frac{\psi'(r)}{\psi(r)}\right)^2 \, \varphi(r)
   \end{equation}
and we consider the polynomial:
        \begin{equation}\label{eq:char}
     P(\lambda):=\lambda^2+(N-2)\lambda+2(N-2) \, .
   \end{equation}
It is readily seen that $P$ admits the negative root:
 \begin{equation} \label{eq:real-eig}
    \lambda_1:=\frac{-(N-2)-\sqrt{(N-2)(N-10)}}{2} \quad (N \ge 10) \, .
\end{equation}
Adapting in a non trivial way the proof of \cite[Lemma 3.1]{hasegawa} to our setting, we prove the following property of the operator $\mathcal L$:

   \begin{lemma} \label{l:comparison-principle} Suppose that $\psi$ satisfies \eqref{A1}, \eqref{A2} and \eqref{eq:A4}. Let $N \ge 10$ and let $0<R<+\infty$, then there exists no function $z\in C^2((0,R])$ such that

   \begin{itemize}
     \item[(i)] $\mathcal Lz>0$ in $(0,R)$;
     \item[(ii)] $z>0$ in $(0,R)$ and $z(R)=0$;
     \item[(iii)] $\psi(r) z'(r)=O(1)$ and $z(r)=o\left((\psi(r))^{-(N-2+\lambda_1)}\right)$ as $r \to 0^+$,
   \end{itemize}
   with $\lambda_1$ as in \eqref{eq:real-eig}.
   \end{lemma}

   \begin{proof}
     Suppose by contradiction that there exists $z\in C^2((0,R])$ satisfying $(i)$-$(iii)$. Let us define
     \begin{equation*} %\label{eq:def-Z}
        Z(r):=(\psi(r))^{\lambda_1} \qquad \text{for any } r>0
     \end{equation*}
     with $\lambda_1$ as in \eqref{eq:real-eig}.
     Differentiating we have that
     \begin{equation*}
        Z'(r)=\lambda_1 (\psi(r))^{\lambda_1-1} \, \psi'(r) \, , \qquad
        Z''(r)=\lambda_1(\lambda_1-1) (\psi(r))^{\lambda_1-2} \, (\psi'(r))^2+\lambda_1(\psi(r))^{\lambda_1-1} \, \psi''(r)
     \end{equation*}
     and hence
     \begin{align} \label{eq:LZ-neg}
        & \mathcal LZ(r)=\lambda_1(\psi(r))^{\lambda_1-1} \, \psi''(r)+(\psi(r))^{\lambda_1-2} (\psi'(r))^2 \,
       P(\lambda_1)\\[7pt]
       \notag & \qquad =\lambda_1(\psi(r))^{\lambda_1-1} \, \psi''(r)<0
     \end{align}
     since $\lambda_1$ is a negative root of the polynomial $P$ and by \eqref{A1}, \eqref{A2}, \eqref{eq:A4} we have
     \begin{equation}\label{eq:A4-rew}
         \frac{\psi'''(r)}{\psi'(r)}>\left(\frac{\psi''(r)}{\psi'(r)}\right)^2 \, , \quad
         \psi''(r)>0 \qquad \text{for any } r>0 \, .
     \end{equation}
        Indeed, the first inequality in \eqref{eq:A4-rew} is equivalent to \eqref{eq:A4} and from it we also have that $\psi'''>0$ which implies $\psi''$ increasing and, in turn, by (A1) we finally obtain $\psi''>0$ in $(0,+\infty)$.

      Now, combining $(i)$, $(ii)$ and \eqref{eq:LZ-neg}, we obtain for any $r\in (0,R)$
      \begin{align*}
         & 0<\mathcal (\mathcal Lz(r)) Z(r)-\mathcal (\mathcal LZ(r))z(r) \\[7pt]
         & \quad
         =\left(z''(r)+(N-1)\frac{\psi'(r)}{\psi(r)} \, z'(r)\right) Z(r)-\left(Z''(r)+(N-1)\frac{\psi'(r)}{\psi(r)} \, Z'(r) \right)z(r)
      \end{align*}
      which is equivalent to
      \begin{equation*}
        \left[(\psi(r))^{N-1} z'(r)Z(r)\right]'-\left[(\psi(r))^{N-1} Z'(r)z(r)\right]'>0 \, .
      \end{equation*}
      In particular, we have the function
      \begin{equation*}
          r \mapsto (\psi(r))^{N-1} \left[z'(r)Z(r)-Z'(r)z(r)\right]
      \end{equation*}
      is increasing in $(0,R)$ and this implies
      \begin{align*}
         & (\psi(R))^{N-1} z'(R)Z(R)>\lim_{r\to 0^+} (\psi(r))^{N-1} \left[z'(r)Z(r)-Z'(r)z(r)\right]  \\[7pt]
         & \quad =\lim_{r\to 0^+} \left[(\psi(r))^{N-2+\lambda_1} \, z'(r)\psi(r)-\lambda_1 (\psi(r))^{N-2+\lambda_1} z(r)\psi'(r)\right]=0
      \end{align*}
      thanks to $(ii)$, $(iii)$, (A1) and the fact that $N-2+\lambda_1>0$.

      We may conclude that $z'(R)>0$ in contradiction with $(ii)$.
   \end{proof}

Thanks to Lemma \ref{l:comparison-principle}, we now prove a uniform estimate for functions defined in \eqref{eq:def-v}.

   \begin{lemma} \label{l:estimate} Suppose that $\psi$ satisfies \eqref{A1}, \eqref{A2} and \eqref{eq:A4} and let $N\ge 10$. Then any function $v$ defined by \eqref{eq:def-v} satisfies
   \begin{equation*}
     v(r)<2\log(\psi'(r)) \qquad \text{for any } r>0 \, .
   \end{equation*}
   \end{lemma}

   \begin{proof}
      Let us define
      \begin{equation} \label{eq:def-V}
         V(r)=2\log(\psi'(r))  \qquad \text{for any } r>0 \, .
      \end{equation}
      If we define $W(r)=V(r)-v(r)$, the statement of the lemma is equivalent to say that $W(r)>0$ for any $r>0$.

      We first observe that by \eqref{eq:def-v}, \eqref{eq:def-V}, (A1) and the fact that $u$ is a radial smooth function in $M$, we have
      \begin{align} \label{eq:W-est}
         & W(r)=2\log(\psi'(r))-u(r)-2\log(\psi(r))+\log[2(N-2)]\sim -2\log(\psi(r)) \to +\infty
      \end{align}
      as $r\to 0^+$.

      Then, recalling the definition of $\lambda_1$ given in \eqref{eq:real-eig}, by \eqref{eq:W-est} and de l'H\^opital's  rule, we infer
      \begin{align*}
         & \lim_{r\to 0^+} (\psi(r))^{N-2+\lambda_1}\, W(r)=\lim_{r\to 0^+} \frac{-2\log(\psi(r))}{(\psi(r))^{-(N-2+\lambda_1)}}=\lim_{r \to 0^+} \, \frac{2}{N-2+\lambda_1} \, (\psi(r))^{N-2+\lambda_1}=0
      \end{align*}
      being $\psi(0)=0$ and $\psi\in C^0([0,+\infty))$ by (A1) and $N-2+\lambda_1>0$ as already observed in the proof of Lemma \ref{l:comparison-principle}. In particular, we have that
      \begin{equation} \label{eq:W=o}
         W(r)=o\left((\psi(r))^{-(N-2+\lambda_1)}\right)  \qquad \text{as } r\to 0^+ \, .
      \end{equation}

      Moreover, differentiating in \eqref{eq:W-est}, using (A1) and exploiting the fact that $u'(0)=0$, we also obtain
      \begin{align*}
         & W'(r)=\frac{2\psi''(r)}{\psi'(r)}-u'(r)-\frac{2\psi'(r)}{\psi(r)}\sim -\frac{2}{\psi(r)}
         \qquad \text{as } r\to 0^+
      \end{align*}
      so that in particular
      \begin{equation} \label{eq:psi-W'}
         \psi(r) W'(r)=O(1) \qquad \text{as } r\to 0^+ \, .
      \end{equation}

      By \eqref{eq:W-est} we know that $W(r)>0$ for any $r$ small enough. We have to prove that actually $W$ is positive in $(0,+\infty)$ and hence we may proceed by contradiction assuming that there exists $R>0$ such that
      \begin{equation} \label{eq:pos-W}
         W(r)>0  \qquad \text{for any } r\in (0,R) \qquad \text{and} \qquad W(R)=0 \, .
      \end{equation}

       We claim that $\mathcal L W>0$ in $(0,R)$ with $\mathcal L$ as in \eqref{eq:def-L}. First of all we observe that by \eqref{eq:A4-rew}, $V$ satisfies
      \begin{align} \label{eq:V-ineq}
         & V''(r)+(N-1)\frac{\psi'(r)}{\psi(r)} \, V'(r)+2(N-2)\left[\frac{e^{V(r)}}{(\psi(r))^2}-\left(\frac{\psi'(r)}{\psi(r)}\right)^2\right]-\frac{2\psi''(r)}{\psi(r)}  \\[7pt]
        & \notag =2\left[\frac{\psi'''(r)}{\psi'(r)}-\left(\frac{\psi''(r)}{\psi'(r)}\right)^2\right]+2(N-2)\frac{\psi''(r)}{\psi(r)}>0
         \qquad \text{for any } r>0 \, .
      \end{align}
      Subtracting \eqref{eq:v-psi} to \eqref{eq:V-ineq}, applying Lagrange Theorem to the exponential function and exploiting \eqref{eq:pos-W}, we infer

      \begin{align} \label{eq:ineq-LW}
         & 0<W''(r)+(N-1) \frac{\psi'(r)}{\psi(r)} \, W'(r)+\frac{2(n-2)}{(\psi(r))^2} \left[e^{V(r)}-e^{v(r)}\right] \\[7pt]
        \notag & \quad <W''(r)+(N-1) \frac{\psi'(r)}{\psi(r)} \, W'(r)+\frac{2(N-2)}{(\psi(r))^2} \, e^{V(r)} [V(r)-v(r)] \\[7pt]
        \notag & \quad =W''(r)+(N-1) \frac{\psi'(r)}{\psi(r)} \, W'(r)+2(N-2)\left(\frac{\psi'(r)}{\psi(r)}\right)^2 W(r)
         =\mathcal L W(r)  \,
      \end{align}
      for any $r\in (0,R)$, thus completing the proof of the claim.

      By \eqref{eq:W=o}, \eqref{eq:psi-W'}, \eqref{eq:pos-W}, \eqref{eq:ineq-LW}, we see that $W$ satisfies conditions $(i)$, $(ii)$, $(iii)$ of Lemma \ref{l:comparison-principle} and the same lemma states that this is impossible.

      We reached a contradiction and this means that $W>0$ in $(0,+\infty)$. The proof of the lemma now follows immediately from the definition of $V$ and $W$.
   \end{proof}

The estimate stated in Lemma \ref{l:estimate} allows us to prove that for $N\ge 10$ solutions of (2.4) are ordered in the following sense:

\begin{lemma} \label{l:ordering} Suppose that $\psi$ satisfies  \eqref{A1}, \eqref{A2} and \eqref{eq:A4} and let $N\ge 10$. Let $\alpha>\beta$ and let $u_\alpha$ and $u_\beta$ two solutions of \eqref{regular_ode_model} satisfying $u_\alpha(0)=\alpha$ and $u_\beta(0)=\beta$. Then $u_\alpha(r)>u_\beta(r)$ for any $r\ge 0$.
\end{lemma}

\begin{proof} Since $\alpha>\beta$ we have that $u_\alpha(r)>u_\beta(r)$ for any $r\ge 0$ small enough. We proceed by contradiction assuming that there exists $R>0$ such that
\begin{equation} \label{eq:ineq-u-u}
    u_\alpha(r)>u_\beta(r) \quad \text{for any } r\in [0,R) \quad \text{and} \quad u_\alpha(R)=u_\beta(R) \, .
\end{equation}
Recalling \eqref{eq:def-v} we may define two functions $v_\alpha$ and $v_\beta$ corresponding to $u_\alpha$ and $u_\beta$ respectively.

By \eqref{eq:ineq-u-u} we obtain
\begin{equation} \label{eq:ineq-v-v}
    v_\alpha(r)>v_\beta(r) \quad \text{for any } r\in [0,R) \quad \text{and} \quad v_\alpha(R)=v_\beta(R) \, .
\end{equation}

By Lemma \ref{l:estimate} we know that
\begin{equation} \label{eq:est-v-v}
    v_\alpha(r)<2\log(\psi'(r)) \quad \text{and} \quad v_\beta(r)<2\log(\psi'(r)) \qquad
    \text{for any } r>0 \, .
\end{equation}

Let us define $w(r)=v_\alpha(r)-v_\beta(r)=u_\alpha(r)-u_\beta(r)$ for any $r\ge 0$. Since both $v_\alpha$ and $v_\beta$ solve \eqref{eq:v-psi}, by \eqref{eq:est-v-v} and Lagrange Theorem we have that
\begin{align} \label{eq:ineq-w}
   & 0=w''(r)+(N-1)\frac{\psi'(r)}{\psi(r)} \, w'(r)+\frac{2(N-2)}{(\psi(r))^2} \left[e^{v_\alpha(r)}-e^{v_\beta(r)} \right] \\[7pt]
  \notag & \quad <w''(r)+(N-1)\frac{\psi'(r)}{\psi(r)} \, w'(r)+\frac{2(N-2)}{(\psi(r))^2} \, e^{v_\alpha(r)}  \, [v_\alpha(r)-v_\beta(r)] \\[7pt]
  \notag  & \quad <w''(r)+(N-1)\frac{\psi'(r)}{\psi(r)} \, w'(r)+2(N-2) \left(\frac{\psi'(r)}{\psi(r)}\right)^2 \, w(r)
    =\mathcal Lw(r)
\end{align}
for any $r\in (0,R)$.

We observe that, by (A1), $w$ trivially satisfies the following two conditions
\begin{equation}\label{eq:asym-w}
   \psi(r) w'(r)=O(1)  \, , \quad w(r)=o\left((\psi(r))^{-(N-2+\lambda_1)}\right) \qquad \text{as } r\to 0^+ \, ,
\end{equation}
being $w=u_\alpha -u_\beta \in C^2([0,+\infty))$.

By \eqref{eq:ineq-v-v}, \eqref{eq:ineq-w} and \eqref{eq:asym-w}, we see that $w$ satisfies conditions $(i)$, $(ii)$ and $(iii)$ of Lemma \ref{l:comparison-principle} and the lemma itself says that this is impossible. We proved that $w>0$ in $[0,+\infty)$ and now the proof of the lemma follows from the definition of $w$.
\end{proof}

\section{Proofs of Theorems \ref{t:stab-inter}, \ref{stability statement}, \ref{t:stability} and \ref{intersection}.}
%\label{6}

\subsection{Proof of Theorem \ref{t:stab-inter}.}
The proof of statements $(i)$ and $(ii)$ follow, respectively, from Lemma \ref{no intesect}  and Lemma \ref{cpoli} below.

\begin{lemma}\label{no intesect}
	Let $N\geq 2$ and let $\psi$ satisfy \eqref{A1}-\eqref{A2}. Let $\alpha$ and $\beta$ be two distinct real numbers and let $u_\alpha$ and $u_\beta$ be the corresponding solutions of \eqref{regular_ode_model}, i.e. $u_\alpha(0)=\alpha$ and $u_\beta(0)=\beta$. If $u_\alpha$ and $u_\beta$ are stable then they do not intersect.
\end{lemma}

\begin{proof}
	Without loss of generality we may assume $\alpha>\beta$; we establish that $u_\alpha(r)>u_\beta(r)$ for any $r>0$. If this is not true then there exists $r_0>0$ such that $u_\alpha(r_0)\leq u_\beta(r_0)$. Now using Lagrange Theorem and Lemma \ref{l4.4} we deduce that for some $\sigma\in (\beta,\alpha)$ we have
	\begin{align*}
		v_\sigma(r_0)=\frac{\partial u}{\partial \alpha}(\sigma, r_0)=\frac{u_\alpha(r_0)-u_\beta(r_0)}{\alpha-\beta}\leq 0
	\end{align*}
and $v_\sigma(0)=1$. Hence there exists $R\in(0,r_0]$ such that $v_\sigma(R)=0$. Then $v_\sigma$ satisfies
\begin{align*}
	-\Delta_{g}v_\sigma=e^{u_\sigma}v_\sigma \text{ in } B_R\quad  \text{ with } v_\sigma=0 \text{ on } \partial B_R.
\end{align*}
Multiplying this equation by $v_\sigma$ and integrating by parts we obtain
\begin{align*}
	\int_{B_R}|\gradg v_\sigma|_g^2\,\dvg-\int_{B_R}e^{u_\sigma} v_\sigma^2\,\dvg=0.
\end{align*}
Let $w_\sigma$ be the trivial extension of $v_\sigma$ to the whole $M$ in such a way that $w_\sigma\in H^1(M)$ and
\begin{align*}
	\int_{M}|\gradg w_\sigma|_g^2\,\dvg-\int_{M}e^{u_\sigma} w_\sigma^2\,\dvg=0.
\end{align*}
Since $u_\alpha$ is a stable solution and $\sigma<\alpha$, by Lemma \ref{l6.7} we deduce that $u_\sigma$ is a stable solution too. This implies that
\begin{align*}
\inf_{v\in H^1(M)\setminus\{0\}}\frac{\int_M |\gradg v|_g^2\, \dvg}{\int_M e^{u_\sigma}|v|^2\,\dvg}\geq 1,
\end{align*}
and $w_\sigma$ attains the infimum. In particular, $w_\sigma$ satisfies the following equation
\begin{align*}
	-\Delta_g w_\sigma = e^{u\sigma}w_\sigma \text{ in }M.
\end{align*}
Now by standard regularity theory $w_\sigma \in C^2(M)$ and satisfies the following equation
\begin{align*}
	-w_\sigma^{\prime\prime}(r)-(N-1)\frac{\psi^\prime(r)}{\psi(r)}w_\sigma^\prime(r)=e^{u_{\sigma}(r)}w_\sigma(r)	\text{ for any } r>0.
\end{align*}
Moreover, by construction $w_\sigma(r)=0$ for $r>R$ and so, by unique continuation, we conclude that $w_\sigma\equiv 0$ in $M$. This is a contradiction.
\end{proof}

\begin{lemma}\label{cpoli}
	Let $N\geq 2$ and let $\psi$ satisfy \eqref{A1}-\eqref{A3}. If $u\in C^2(M)$ is an unstable solution of \eqref{regular_ode_model}, then there exists no solution $v\in C^2(M)$ of \eqref{regular_ode_model} such that $v>u$ in M.
\end{lemma}

\begin{proof}
Suppose that there exists a solution $v\in C^2(M)$ of \eqref{regular_ode_model} such that $v>u$. Let $w=v-u$, so $w>0$ in $M$. Using the Lagrange Theorem we have
\begin{align*}
	-\Delta_{g}w=e^v-e^u\geq e^uw \text{ in } M \, .
\end{align*}
Now taking $\varphi \in C^\infty_c(M)$ and multiplying in the above inequality by $\frac{\varphi^2}{w}$, integrating by parts and using Cauchy Schwarz and Young inequalities, we infer
\begin{align*}
	\int_M e^u \varphi^2 \dvg & \le -\int_M \left\langle \nabla_g w,\nabla_g \left(\frac{\varphi^2}{w}\right) \right\rangle_g\,  \dvg\\&=2\int_M \left\langle \nabla_gw, \frac{\varphi}{w} \nabla_g \varphi \right\rangle_g\dvg - \int_M \frac{\varphi^2}{w^2}|\nabla_g w|_g^2 \, \dvg \leq \int_M |\nabla_g \varphi |_g^2 \, \dvg.
\end{align*}
This implies that $u$ is a stable solution and gives a contradiction.
\end{proof}

\subsection{Proofs of Theorems \ref{stability statement} and \ref{t:stability}.}
The proofs of Theorem  \ref{stability statement} - $(i)$ and $(ii)$ follow by Lemma \ref{t6.2} and Lemma \ref{l6.8} which give that $\mathcal{S}=(-\infty,\eta]$. By Lemma \ref{l6.3} we know that $\eta \geq \log(\lambda_1(M))$ while Lemma \ref{l:eta>} allows improving this bound when $\psi/\psi'$ is not integrable. \par
The proof of Theorem \ref{t:stability} follows arguing by contradiction. Indeed, suppose that there exists a radial solution $u\in C^2(M)$ of \eqref{regular_ode_model} with $N\geq 10$ which is unstable. Let $v\in C^2(M)$ be a radial solution of  \eqref{regular_ode_model} such that $v(0)>u(0)$. Then, by Lemma \ref{l:ordering} we have that $v>u$ on $M$, in contradiction with Lemma \ref{cpoli}.

\subsection{Proof of Theorem \ref{intersection}.}
The proofs of Theorem \ref{intersection} -$(i)$ and $(iii)$ follow by combining Theorem \ref{stability statement}-$(i)$ and Theorem \ref{t:stability} with Lemma \ref{no intesect} . Instead, the proof of Theorem \ref{intersection}-$(ii)$ follows by combining Theorem \ref{stability statement}-$(ii)$ with Lemma \ref{cpoli}.

%
%
%\begin{lemma}
%	Let $\psi$ satisfy \eqref{A1}-\eqref{A3} and $\eta<+\infty$. Then for $\alpha$, $\beta>\eta$ with $\alpha\neq \beta$, $u_\alpha$ and $u_\beta$ have at least one intersection in $(0,+\infty)$.
%\end{lemma}
%\begin{proof}
%Without loss of generality we can assume $\alpha>\beta$. Since $\beta>\eta$ then $u_\beta$ is an unstable solution and so, using Lemma \ref{l6.7}, we get that $u_\alpha$ is also an unstable solution. If there is no intersection then we will have $u_\alpha(r)>u_\beta(r)$ for any $r\in (0,+\infty)$, contradicting Lemma \ref{cpoli}. Hence, $u_\alpha$ and $u_\beta$ have at least one intersection in $(0,+\infty)$.
%\end{proof}

	\section*{Appendix: some (well) known facts in the Euclidean case}
    Consider the equation
    \begin{equation} \label{eq:Gelfand}
       -\Delta u= e^u  \qquad \text{in } \R^N
    \end{equation}
    and let $u$ be a radial regular solution of \eqref{eq:Gelfand}. Letting $\alpha=u(0)$, then $u$ solves \eqref{regular_ode_model} with $\psi(r)=r$. Following \cite{joseph-73}, we consider the function $v(r)=u(r)+2\log r-\log[2(N-2)]$ (i.e., \eqref{eq:def-v} with $\psi(r)=r$) which satisfies the equation
    \begin{equation*}
      -v''(r)-\frac{N-1}{r} \, v'(r)=\frac{2(N-2)}{r^2} [e^{v(r)}-1] \quad (r>0) \, .
    \end{equation*}
   Now, let $w(t)=v(e^t)$ so that $w$ solves the autonomous equation
   \begin{equation} \label{eq:auto}
     w''(t)+(N-2)w'(t)+2(N-2) [e^{w(t)}-1]=0  \quad (t\in \R) \, .
   \end{equation}
  Following \cite{joseph-73} one can reduce equation \eqref{eq:auto} into an \emph{autonomous} system in the plane $(y,z)$ admitting the unique stationary point $(0,0)$ where we put $z(t)=w(t)$ and $y(t)=w'(t)$. Clearly, the system is given by
   \begin{equation}\label{eq:system}
    \begin{cases}
     y'(t)=-(N-2)y(t)-2(N-2)[e^{z(t)}-1] \, , \\[5pt]
     z'(t)=y(t) \, .
    \end{cases}
   \end{equation}
   The behavior and, in turn, the stability of radial solutions to \eqref{eq:Gelfand} depend on the nature of the stationary point $(0,0)$ of \eqref{eq:system} and, in particular, after linearization at $(0,0)$, on the nature of the eigenvalues of the matrix
    \begin{equation} \label{eq:matrix}
      \begin{pmatrix}
        -(N-2) & -2(N-2) \\
        1      & 0
      \end{pmatrix} \, .
    \end{equation}
   We observe that the characteristic polynomial of the matrix \eqref{eq:matrix} is exactly the polynomial $P$ given in \eqref{eq:char}. For $3\le N \le 9$ it admits two complex conjugate eigenvalues while for $N \ge 10$ it admits two negative eigenvalues which are coinciding if $N=10$ and distinct if $N \ge 11$:
$$
    \lambda_1=\frac{-(N-2)-\sqrt{(N-2)(N-10)}}{2} \, , \ \ \  \lambda_2=\frac{-(N-2)+\sqrt{(N-2)(N-10)}}{2} \quad (N \ge 10) \, .
$$
The above computations highlight a change in the nature of solutions when passing from low ($N\leq 9$) to high dimensions ($N\geq 10$). Starting from this observation, see e.g., \cite[Theorem 1.1]{tello-06}, it can be proved that radial smooth solutions intersect themselves infinitely many times if $3\leq N\leq 9$ and do not intersect if $N\geq 10$. Furthermore, for $N \le 9$ all regular solutions of \eqref{eq:Gelfand} are unstable while for $N\ge 10$ radial regular solutions are stable, see \cite{farina2} and reference therein.

\bigskip

 \textbf{Acknowledgments.} The first two authors are members of the Gruppo Nazionale per l'Analisi Matematica, la Probabilit\`a e le loro Applicazioni (GNAMPA) of the Instituto Nazionale di Alta Matematica (INdAM) and are partially supported by the PRIN project 201758MTR2: ``Direct and inverse problems for partial differential equations: theoretical aspects and applications'' (Italy).
 The first two authors acknowledge partial financial support from the INdAM - GNAMPA project 2022 ``Modelli del 4 ordine per la dinamica di strutture ingegneristiche: aspetti analitici e applicazioni''.

The second author acknowledges partial financial support from the research project ``Metodi e modelli per la matematica e le sue applicazioni alle scienze, alla tecnologia e alla formazione'' Progetto di Ateneo 2019 of the University of Piemonte Orientale ``Amedeo Avogadro''.
  The third author is partially supported by the INSPIRE faculty fellowship (IFA17-MA98). The fourth author is supported in part by
 National Theoretical Science Research Center Operational Plan V-Mathematics Field (2/5) (Project number 111-2124-M-002-014-G58-01).
 
 The authors are grateful to the anonymous referee for bringing to their attention two articles containing interesting applications to fluid dynamics of equation \eqref{main_eq} and related semilinear equations.
	
%%%%%%%%%%%%%%%%%%%%%%%%%%%%%%%%%%%%%%%%%%%%%%%%%%%%%%%%%%%%%%%%%%%%%%%%%%%%%%%%%%%

\end{document}